\documentclass[twoside,11pt]{article}

\usepackage{jmlr2e}

\usepackage{algorithm}
\usepackage{algorithmic}
\usepackage{lipsum}
\usepackage{essay-def-JMLR}
\usepackage{clrscode}
\usepackage{caption}
\usepackage{appendix}
\usepackage{mathrsfs}
\usepackage{bm}
\usepackage{booktabs}
\usepackage{url}
\usepackage{graphicx,epstopdf}
\usepackage[caption=false]{subfig}
\usepackage{multirow}
\usepackage{textcomp}

\ShortHeadings{SDC makes first-order methods faster}{Wang, Jia and Wen}
\firstpageno{1}

\begin{document}

\title{The Search direction Correction makes first-order methods faster}

\author{\name Yifei\ Wang \email zackwang24@pku.edu.cn \\
       \addr School of Mathematical Sciences\\
       Peking University , CHINA\\
       \AND
       \name Zeyu\ Jia \email jiazy@pku.edu.cn \\
       \addr School of Mathematical Sciences\\
       Peking University , CHINA\\
       \AND
       \name Zaiwen\ Wen \email wenzw@pku.edu.cn \\
       \addr Beijing International Center for Mathematical Research\\ 
       Peking University, CHINA}
\editor{}
\maketitle

\begin{abstract}
The so-called fast inertial relaxation engine is a first-order method for unconstrained smooth optimization problems. It updates the search direction by a linear combination of the past search direction, the current gradient and the normalized gradient direction. We explore more general combination rules and call this generalized technique as the search direction correction (SDC). SDC is extended to composite and stochastic optimization problems as well. Deriving from a second-order ODE, we propose a fast inertial search direction correction (FISC) algorithm as an example of methods with SDC. We prove the $\mcO(k^{-2})$ convergence rate of FISC for convex optimization problems. Numerical results on sparse optimization, logistic regression as well as deep learning demonstrate that our proposed methods are quite competitive to other state-of-the-art first-order algorithms. 
\end{abstract}


\section{Introduction}
We take the following optimization problem into consideration
\begin{equation}\label{problem}
	\min_{\mfx\in\mathbb{R}^{n}}f(\mfx) = \psi(\mfx) + h(\mfx),
\end{equation}
where $\psi$ is a smooth function and $h$ is a possibly non-smooth convex function. In machine learning, $\psi$ often has the form
\begin{equation}\label{form}
	\psi(\mfx) = \frac{1}{N}\sum_{i=1}^{N}\psi_{i}(\mfx),
\end{equation}
where $\psi_{i}$ is the prediction error to the $i$-th sample. Since the dimension of the variable $x$ and the number of samples $N$ are often extremely huge, first-order and/or stochastic algorithms are frequently used for solving \eqref{problem}.

\par First-order algorithms only use the information of the function value and the gradient. The vanilla gradient descent method is the simplest algorithm with convergence guarantees. Adding momentum to the current gradient has been an efficient technique to accelerate the convergence. This type of algorithms includes the Nesterov accelerated method \citep{Nesterov}, the Polyak heavy-ball method \citep{ito}, and the nonlinear conjugate gradient method \citep{ncgm}. Except the last one, these methods can be extended to cases where $h$ is non-smooth, by replacing the gradient with the so-called proximal gradient. Meanwhile, \citet{iloco} proved that first-order algorithms cannot achieve convergence rate better than $\mathcal{O}(1/k^{2})$. In this way, the convergence rate of the Nesterov accelerated method matches this lower bound exactly.
\par Lately, a new technique borrowed from ODE and dynamical system has been used to analyze the behavior of optimization algorithms. \citet{adefm} analyzed several ODEs which correspond to different types of Nesterov accelerated methods when the step size converges to zero. With specifically designed Lyapunov functions, they obtained proportional convergence rate for these ODEs and for Nesterov accelerated methods. \citet{avpoa} and \citet{alaom} generalized this technique to a broader class of first-order algorithms. \citet{drkda} proposed a different type of Lyapunov function and obtained a convergence competitive to Nesterov accelerated methods.
\par The stochastic gradient descent method (SGD) is the stochastic version of
the vanilla gradient descent method. However, SGD may suffer from the large variance of stochastic gradients during its iterations. To tackle this problem, SVRG \citep{svrg},  SAG \citep{sag} and SAGA \citep{saga} introduce variance reduction techniques and achieve acceleration compared to SGD.
\par Recently, an optimization algorithm called fast inertial relaxation engine (FIRE) \citep{FIRE} is proposed for finding the atomic structures with the minimum potential energy. Involving an extra term of the velocity correction along the gradient direction with the same magnitude of the current velocity, and adopting a carefully designed restarting criterion, FIRE can achieve better performance than the conjugate gradient method. It is even competitive to the limited-memory BFGS \citep{LBFGS} in several test cases. However, neither the choice of molecular dynamics integrator is specified nor the convergence rate is given in the work of \citet{FIRE}.
\par Motivated by first-order algorithms and FIRE, we introduce a family of first-order methods with the search direction correction (SDC) and propose the fast inertial search direction correction (FISC) algorithm. Our contributions are listed as follows:
\begin{itemize}
	\item We adapt FIRE in molecular dynamics to solve general smooth and nonsmooth optimization problems. We explore more general combination rules of updating search direction in FIRE and generalize it into a framework of first-order methods with SDC. We allow more choices for step sizes, such as applying line search technique to find a step size satisfying the Armijo conditions or the nonmonotone Armijo conditions. The basic restarting criterion ensures the global convergence for methods with SDC. Furthermore, SDC is extended to composite optimization and stochastic optimization problems. 
	\item Second-order ODEs of methods with SDC in continuous time are derived via taking the step size to zero. Through the discretization of ODEs, our algorithms are recovered. By constructing a Lyapunov function and analyzing its derivative, we prove that the ODE corresponding to FISC has the convergence rate of $O(1/t^2)$ on smooth convex optimization problems. We also build a discrete Lyapunov function for FISC in the discrete case. On composite optimization problems, FISC is proven to have the $\mathcal{O}(1/k^{2})$ convergence rate.  
	\item Our algorithms are tested on sparse optimization, logistic regression and deep learning. Numerical experiments indicate that our algorithms are quite competitive to other state-of-the-art first-order algorithms. 
\end{itemize}

\subsection{Organization}
This paper is organized as follow. We present the update rule of methods with
SDC including FISC in Section \ref{sec:SDC}. In Section \ref{sec:ode}, the ODE
perspective of FISC is used to provide a necessary condition for the
convergence. The global convergence of methods with SDC and the convergence rate
of FISC are discussed in Section \ref{sec:converge}. Finally, in Section \ref{sec:num}, we present numerical experiments to compare FISC, FIRE and other first-order algorithms.

\subsection{Preliminaries}
We use standard notations throughout the paper. $\|\cdot\|$ is the standard Euclidean norm and $\la\cdot, \cdot \ra$ is the standard Euclidean inner product. $\mcF_L$ stands for the class of convex and differentiable functions with $L$-Lipschitz continuous gradients. $\mcF$ represents the class of convex and differentiable functions. $\mbR^+$ is the collection of non-negative real number. $[N]$ denotes $\{1, 2, \dots N\}$.

\section{The framework of SDC}
\label{sec:SDC}
In this section, we introduce the framework of first-order methods with SDC to solve smooth optimization problems \eqref{problem} with $h=0$. SDC is extended to composite optimization problems, stochastic optimization problems and deep learning later. 
\subsection{A family of first-order methods with SDC}
In this subsection, we focus on solving smooth optimization problems \eqref{problem} with $h=0$. It involves two sequences of parameters $\{\beta_k\}_{k=1}$ and $\{\gamma_k\}_{k=1}$ and introduces a velocity $\mfu$ as a search direction to update $\mfx$. 

We start with an initial guess $\mfx_0$ and an initial velocity $\mfu_0=0$. In the beginning of the $(k+1)$-th iteration, we determine whether $\mfu_k$ is a descent direction by introducing a restarting criterion 
\begin{equation}
\label{rst_crit}
\varphi_k=\la -\nabla f(\mfx_k), \mfu_k\ra\geq 0,
\end{equation}
If this criterion holds, we update
\begin{equation}
\label{SDC-upd}
\mfu_{k+1}=(1-\beta_k)\mfu_k-\gamma_k\frac{\|\mfu_k\|}{\|\nabla f(\mfx_k)\|} \nabla f(\mfx_k)- \nabla f(\mfx_k).
\end{equation}
When $k=0$, we directly have $\mfu_1=-\nabla f(\mfx_0)$ given $\mfu_0=0$, so $\beta_0$ and $\gamma_0$ need not be specified. We further require $\beta_k$ and $\gamma_k$ to satisfy
\begin{equation}
\label{bg_cond}
0\leq\beta_k\leq1, \quad 0\leq\gamma_k\leq1.
\end{equation}
Then we update $\beta_{k+1}$ and $\gamma_{k+1}$ as follows.
\begin{itemize}
\item In \textbf{FIRE} \citep{FIRE}, they are updated by
$$
\gamma_{k+1}=\beta_{k+1}=d_\beta\beta_k,
$$
where $0<d_\beta<1$ is a parameter. The initial value of $\{\beta_k\}$ is set to $\beta_1=1$ and $d_\beta$ is given by $d_\beta=0.99$. 
\item In \textbf{FISC}, $\beta_k$ and $\gamma_k$ are parameterized with $l_k$, i.e.,
\begin{equation}
\label{FISC-para}
\beta_k=\frac{r}{l_k-1+r}, \quad \gamma_k=\frac{r-3}{l_k-1+r},
\end{equation}
where $r\geq 3$ and $\{l_k\}$ is a sequence of parameters with an initial value of $l_1=1$. We update $l_{k+1}=l_k+1$.
\end{itemize}
If the criterion \eqref{rst_crit} is not met, we restart the system by resetting $\mfu_{k+1}, \beta_{k+1}$ and $\gamma_{k+1}$ as:
\begin{gather}
\label{SDC-rst}
\mfu_{k+1}=-\nabla f(\mfx_k),\\
\label{bg_rst}
\beta_{k+1}=\beta_1, \gamma_{k+1}=\gamma_1.
\end{gather}
Specifically, in FISC, we reset $l_{k+1}=l_1$.

Then, we calculate the step size $s_k$. Either of the following choices of $s_k$ is acceptable:
\begin{enumerate}
\item[(i)] Fix the step size $s_k=s_0$.
\item[(ii)] Perform a backtracking line search to find a step size $s_k$ that satisfies the Armijo conditions:
\begin{equation}
\label{armijo}
f(\mfx_k-s_k\mfu_{k+1})\leq f(\mfx_k)-\sigma s_k\la \mfu_{k+1}, \nabla f(\mfx_k)\ra,
\end{equation}
where $0<\sigma <1$ is a parameter and $s_k=\bar s_k \rho^{h_k}$. Here $\bar s_k>0$ is the trial step and $h_k$ is the largest number such that \eqref{armijo} holds.
\item[(iii)] Perform a nonmonotone line search \citep{anlst} to find a step size $s_k$ that satisfies nonmonotone Armijo conditions:
\begin{equation}
\label{equ:nonbls}
f(\mfx_k-s_k \mfu_{k+1})\leq C_k-\frac{s_k}{2}\la \mfu_{k+1}, \nabla f(\mfx_k)\ra,
\end{equation}
where $s_k=\bar s_k \rho^{h_k}$. Here $\bar s_k>0$ is the trial step and $h_k$ is the largest number such that \eqref{equ:nonbls} holds. $C_k$ and $Q_k$ are updated as:
$$
Q_{k+1} = \eta_kQ_k+1, \quad C_{k+1}=(\eta_kQ_kC_k+f(\mfx_{k+1}))/Q_{k+1},
$$
with initial values $C_0=f(\mfx_0), Q_0=1$. $\eta_k$ is selected from $[\eta_{min},\eta_{max}]$. The existence of $s_k$ is proved in Subsection \ref{ssec:glo}. 
\end{enumerate}

After calculating the step size $s_k$, we update
\begin{equation}
\label{x_upd}
\mfx_{k+1}=\mfx_k+s_k\mfu_{k+1}.
\end{equation}
Then, we replace $k$ by $k+1$ and check whether convergence criteria are satisfied. A family of first-order methods with SDC is given in Algorithm \ref{FISC}.

\begin{algorithm}[!htp]
\caption{A family of first-order methods with SDC}
\label{FISC}
\begin{algorithmic}[1]
\REQUIRE initial guess $\mfx_0$, initial value $\mfu_0=0$, other required parameters.

\STATE set $k=0$, fix step size $s_0$ or calculate it by the line search.
\WHILE{\textit{Convergence criteria are not met} or $k< N_{max}$}
\STATE Calculate $\varphi_k$ by \eqref{rst_crit}.
\IF{$\varphi_k\geq0$}
\STATE Compute $\mfu_{k+1}$ by \eqref{SDC-upd} and update $\beta_{k+1}, \gamma_{k+1}$.
\ELSE
\STATE Set $\mfu_{k+1}$ by \eqref{SDC-rst} and reset $\beta_{k+1},\gamma_{k+1}$.
\ENDIF
\STATE Fix step size $s_k$ or calculate it using line search techniques.
\STATE Update $\mfx_{k+1}$ by \eqref{x_upd}, $k\to k+1$.
\ENDWHILE
\RETURN $\mfx_k$
\end{algorithmic}
\end{algorithm}

Compared to the original FIRE \citep{FIRE}, we make several adaptions: 
\begin{itemize}
\item specify the symplectic Euler scheme as the MD integrator;
\item remove the ``latency'' time of MD steps before accelerating the system;
\item apply line search techniques in calculating step sizes;
\item rescale the MD step size $\Delta t_k$ by $s_k=(\Delta t_k)^2$ and rescale the velocity $\mfv_k$ in MD to $\mfu_k= \mfv_k/\sqrt{s_k}$.
\end{itemize}

\subsection{A variant of FISC}
In this subsection, we introduce FISC-ns, a variant of FISC. Detailed derivation of FISC and FISC-ns is shown in Section \ref{sec:ode}. In FISC-ns, $\mfu_k$ is replaced by an auxiliary variable $\mfy_k$ and $\{l_k\}_{k=1}$ in FISC-ns remains the same. We start with $\mfx_0=\mfx_{-1}$. Given $\mfx_k$ and $\mfx_{k-1}$, the restarting criterion uses the quantity 
$$
\varphi_k=\la -\nabla f(\mfx_k), \mfx_k-\mfx_{k-1}\ra.
$$
If $\varphi_k\geq0$, we compute $\mfy_{k}$ by
$$
\mfy_{k}=\mfx_k+\frac{l_k-1}{l_k-1+r}(\mfx_k-\mfx_{k-1})-\frac{r-3}{l_k-1+r}\frac{||\mfx_k-\mfx_{k-1}||}{||\nabla f(\mfx_k)||}\nabla f(\mfx_k). 
$$
The step $s_k$ is calculated at $\mfy_k$ using the direction $-\nabla f(\mfy_k)$. We then update
\begin{equation}
\label{y_upd}
\mfx_{k+1}=\mfy_k-s_k\nabla f(\mfy_k),
\end{equation}
and update $l_{k+1}$. Otherwise, we calculate the step size $s_k$ at $\mfx_k$ using the direction $-\nabla f(\mfx_k)$. Then $\mfx_{k+1}$ is updated by
\begin{equation}\label{xk_upd}
\mfx_{k+1}=\mfx_k-s_k\nabla f(\mfx_k),
\end{equation}
and we reset $l_{k+1}=l_1$.

If no restarting criterion is triggered and the step size is fixed to be $s$, FISC updates
\begin{equation}
\label{FISC-upd}
\mfx_{k+1}=\mfx_k+\frac{l_k-1}{l_k-1+r}(\mfx_k-\mfx_{k-1})-\frac{r-3}{l_k-1+r}\frac{||\mfx_k-\mfx_{k-1}||}{||\nabla f(\mfx_k)||}\nabla f(\mfx_k)+s\nabla f(\mfx_k),
\end{equation}
while FISC-ns updates 
\begin{equation}
\label{FISC-ns-upd}
\mfx_{k+1}=\mfx_k+\frac{l_k-1}{l_k-1+r}(\mfx_k-\mfx_{k-1})-\frac{r-3}{l_k-1+r}\frac{||\mfx_k-\mfx_{k-1}||}{||\nabla f(\mfx_k)||}\nabla f(\mfx_k)+s\nabla f(\mfy_k).
\end{equation}
In Subsection \ref{sec:converge}, we prove that with the update rule of FISC-ns \eqref{FISC-ns-upd}, FISC-ns has an $\mcO(k^{-2})$ convergence rate. With $r>3$, FISC-ns has to calculate the gradient twice in updating $\mfx_{k+1}$, which may be computationally costly. On the other hand, the update rule of FISC \eqref{FISC-upd} can be viewed as an approximation of the update rule of FISC-ns \eqref{FISC-ns-upd} and it only evaluates the gradient once in each iteration. In short, FISC-ns has better theoretical explanations and the performance of FISC is better in practice.

\subsection{SDC for other optimization problems}
\label{sec:mod}
\subsubsection{Composite optimization problems}
\label{sec:ffcop}
Consider the composite optimization problem \eqref{problem}, where $\psi\in\mcF_L$. Given the convex function $h$ and the step size $s>0$, we define the proximal mapping of $h$ as 
$$
\mathrm{prox}^s_{h}(\mfx)=\arg\min_\mfz\left(\frac{1}{2s}\|\mfz-\mfx\|^2+h(\mfz)\right).
$$
Based on the proximal mapping, the proximal gradient is defined by
$$
G_s(\mfx)=\frac{\mfx-\mathrm{prox}_{h}^s(\mfx-s\nabla \psi(\mfx))}{s}.
$$
Here we present two ways to modify SDC for composite optimization problems. The first way is to use the proximal gradient. We simply replace the gradient $\nabla f(\mfx)$ in \eqref{SDC-upd}  by the proximal gradient $G_s(\mfx)$. In $(k+1)$-th iteration, the step size $s_k$ is fixed or calculated at $\mfx_k$ for the proximal gradient, using line search techniques. The basic restarting criterion uses the quantity
\begin{equation}
\label{prox_c}
\varphi_k=\la \mfu_k, - G_{s_k}(\mfx_k)\ra.
\end{equation}
If $\varphi_k\geq0$, then we will update $\mfu_{k+1}$ by

\begin{equation}
\label{SDC-PG-upd}
\mfu_{k+1}=(1-\beta_k)\mfu_k-\gamma_k\frac{\|\mfu_k\|}{\|G_{s_k}(\mfx_k)\|} G_{s_k}(\mfx_k)- G_{s_k}(\mfx_k).
\end{equation}
Otherwise, $\mfu_{k+1}$ is reset by 

\begin{equation}
\label{SDC-PG-rst}
\mfu_{k+1}=- G_{s_k}(\mfx_k),
\end{equation}
and $\beta_{k+1}, \gamma_{k+1}$ are reset using \eqref{bg_rst}. Then $\mfx_{k+1}$ is calculated by \eqref{x_upd}.

The second way is to use the proximal mapping. We introduce an auxiliary variable $\mfy_k\in\mbR^n$ and start with $\mfx_0=\mfx_{-1}$. Given $\mfx_k$ and $\mfx_{k-1}$, the restarting criterion uses the following quantity:
$$
\varphi_k=\la \mfx_k-\mfx_{k-1}, - G_{s_k}(\mfx_k)\ra.
$$
If $\varphi_k\geq0$, the step size $s_k$ is fixed or calculated at $\mfx_k$ for the proximal gradient using similar methods. $\mfy_{k}$ is updated by
 
\begin{equation}
\label{SDC-PM-upd}
\mfy_{k}=\mfx_k+(1-\beta_k)(\mfx_{k}-\mfx_{k-1})-\gamma_k\frac{\|\mfx_{k}-\mfx_{k-1}\|}{\|G_{s_k}(\mfx_k)\|} G_{s_k}(\mfx_k).
\end{equation} 
Then we fix the step size $\bar s_k$ or calculate it at $\mfy_k$ for the proximal mapping, compute
\begin{equation}
\label{yy_upd}
\mfx_{k+1}=\mfy_{k}-{\bar s_k}G_{\bar s_k}(\mfy_k),
\end{equation}
and update $\beta_{k+1},\gamma_{k+1}$. Note that $\mfx_{k+1}$ is the proximal mapping of $\mfy_k$, i.e., $\mfx_{k+1}=\mathrm{prox}_{h}^{\bar s_k}(y_k)$. 

Otherwise, we fix $s_k$ or calculate it at $\mfx_k$ for the proximal mapping, update 
\begin{equation}
\label{yy_rst}
\mfx_{k+1}=\mfx_{k}-s_kG_{s_k}(\mfx_k),
\end{equation}
and reset $\beta_{k+1},\gamma_{k+1}$ by \eqref{bg_rst}. 

By taking $\beta_k=\frac{r}{l_k-1+r}$ and $\gamma_k=\frac{r-3}{l_k-1+r}$ in \eqref{SDC-PM-upd}, we obtain FISC-PM. With $r=3$ in FISC-PM, FISTA \citep{FISTA} can be recovered. FISC-ns is a specific version of FISC-PM with the non-smooth part $h=0$ in \eqref{problem}. 

\subsubsection{Stochastic composite optimization problems}
Consider the stochastic composite optimization problem \eqref{problem},
where $\psi$ has the form \eqref{form} and $\psi_i\in \mcF_L$. In each iteration, we generate stochastic approximations of the gradient via selecting sub-samples $\mcT_k\subset [N]$ uniformly at random. That is, the mini-batch stochastic oracle is obtained as follows:
\begin{equation}
\label{sto_ora}
\nabla \psi^{(k)}(\mfx)=\frac{1}{|\mcT_k|}\sum_{i\in\mcT_k}\nabla \psi_i(\mfx).
\end{equation}

Motivated by \cite{pSVRG}, we also adopt the variance reduced version of stochastic gradient. With an extra parameter $m\in\mbN$, the stochastic oracle can be as follows:
\begin{equation}
\label{ora_vr}
\left\{
\begin{aligned}
&\mathrm{If}\,\,\, k\,\,\mathrm{mod}\,\,m=0\quad \mathrm{then} \quad \mathrm{set}\,\, \tilde \mfx=\mfx_k \,\,\mathrm{and}\, \, \mathrm{calculate}\,\, \nabla \psi(\tilde \mfx).\\
&\mathrm{Compute} \,\,\nabla \psi^{(k)}(\mfx_k)=\frac{1}{|\mcT_k|}\sum_{i\in\mcT_k}(\nabla \psi_i(\mfx)-\nabla \psi_i(\tilde\mfx))+\nabla \psi(\tilde \mfx).
\end{aligned}\right.
\end{equation}
Here $k$ is the current iteration number and $m$ is the number of iterations after which the full gradient $\nabla \psi$ is evaluated at the auxiliary variable $\tilde \mfx$. Similar to \citep{assnm}, this additional noise-free information is stored and utilized in the computation of the stochastic oracles in the following iterations.

Then, the proximal stochastic gradient is calculated by
$$
G_{s_k}(\mfx)=\frac{\mfx-\mathrm{prox}_{h}^{s_k}(\mfx-s\nabla \psi^{(k)}(\mfx))}{s_k}.
$$
The criterion $\varphi_k\geq0$ is evaluated using \eqref{prox_c}. If it is satisfied, we update the velocity $\mfu_{k+1}$ by \eqref{SDC-PG-upd}.  Otherwise, we reset $\mfu_{k+1}$, $\beta_{k+1},\gamma_{k+1}$ by \eqref{SDC-PG-rst} and \eqref{bg_rst}.  SDC for the stochastic optimization can be obtained by setting the non-smooth part $h=0$ in \eqref{problem}.

\subsection{SDC in deep learning}
We also adapt SDC to the deep learning setting. Because the target function is highly nonconvex, we make the following changes in updating rules. In the $(k+1)$-th iteration, we first calculate the ``momentum and gradient update'' on $\mfu_k$ as follows:
$$
\tilde\mfu_k=\alpha \mfu_k-\mfg_k,
$$
where $0<\alpha<1$ is a parameter and $\mfg_k$ is the stochastic gradient of $f$ evaluated at $\mfx_k$ through back-propagation. The basic restarting criterion uses 
$$
\varphi_k=\la\tilde \mfu_k, -\mfg_k\ra. 
$$
If $\varphi_k\geq0$, we calculate $\mfu_{k+1}$ by correcting $\tilde \mfu_k$ to
\begin{equation}\label{uk_dl}
\mfu_{k+1}=(1-\beta_k)\tilde \mfu_k-\gamma_k\frac{\|\tilde \mfu_k\|}{\|\mfg_k\|} \mfg_k.
\end{equation}
Otherwise, we set 
$$
\mfu_{k+1}=-\mfg_k.
$$
Then, $\mfx_{k+1}$ is updated by \eqref{xk_upd}. Note that if we simply uses $\tilde \mfu_k$ or $-\mfg_k$ as $\mfu_{k+1}$, then we will get SGD with momentum or vanilla SGD. \eqref{SDC-upd} performs SDC on $\mfu_k$ while \eqref{uk_dl} performs SDC on $\tilde\mfu_k$.

\subsection{The comparison with other first-order methods}
In this subsection, we compare first-order methods with SDC with the Nesterov's accelerated method with restarting \citep{arfag}, the heavy-ball method \citep{ito} and the nonlinear Conjugate Gradient (CG) method \citep{ncgm}.

\subsubsection{The Nesterov's accelerated method with restarting}
Suppose that the step size is fixed, i.e., $s_k=s$. Taking the limiting process $s\to0$, the restarting criterion \eqref{rst_crit} essentially keeps $\la \dot \mfx, \nabla f(\mfx)\ra $ negative. This coincides with the heuristic in \citep{arfag}, where they proposed a procedure termed as gradient restarting for the Nesterov's accelerated method. Its update rule is given by:
\begin{equation}
\label{Nesterov-upd}
\left\{
\begin{aligned}
&\mfx_k=\mfy_{k-1}-s\nabla f(\mfy_{k-1}),\\
&\mfy_k=\mfx_k+\frac{k-1}{k+2}(\mfx_k-\mfx_{k-1}).
\end{aligned}
\right.
\end{equation}
The algorithm restarts with $\mfx_0=\mfy_0:=\mfx_k$ and resets $k=0$, whenever 
$$
\la \nabla f(\mfy_k),  \mfx_k-\mfx_{k-1}\ra>0.$$
We shall note that this coincides with FISC-ns when $r=3$. If one takes step size $s\to 0$, this restarting criterion also keeps $\la \nabla f(\mfx), \dot \mfx\ra$ non-positive along the trajectory, and resets $k$ to prevent the coefficient $(k-1)/(k+2)$ from steadily increasing to $1$. 

\subsubsection{The heavy-ball method}
Consider the case where no restarting criterion is triggered and the step size $s_k$ is fixed. The update rule of velocity $\mfu_{k+1}$ in the heavy-ball method \citep{ito}:
\begin{equation}
\label{HB-upd}
\mfu_{k+1}=\beta^{(\text{HB})} \mfu_k-\nabla f(\mfx_k).
\end{equation}
Then, the heavy-ball method update $\mfx_{k+1}$ in the same way as \eqref{x_upd}. The coefficient of $\mfu_k$ in the heavy-ball method is a constant $\beta^{(\text{HB})}$, while $\beta_k$ in FIRE decay exponentially and $\beta_k$ in FISC decay linearly with regard to $k$. Compared to the Heavy-ball method, FIRE/FISC introduce an extra term $\gamma_k\frac{\|\mfu_k\|}{\|\nabla f(\mfx_k)\|} \nabla f(\mfx_k)$ in updating $\mfu_{k+1}$. 

\subsubsection{The non-linear CG method}
In this case, we obtain step size $s_k$ by line search techniques and the update rule of search direction $\mfu_{k+1}$ reads
\begin{equation}
\label{CG-upd}
\mfu_{k+1}=\beta^{(CG)}_k \mfu_k-\nabla f(\mfx_k).
\end{equation}
If $\mfu_{k}$ does not have the descent property, i.e., $\la -\nabla f(\mfx_k), \mfu_k\ra<0$, CG will restart by setting $\mfu_{k+1}=-\nabla f(\mfx_k)$. In FIRE, when $\varphi_k=\la -\nabla f(\mfx_k), \mfu_k\ra$ in the restarting criterion is negative, $\mfu_{k+1}$ is reset in \eqref{SDC-rst} as same as CG.  Though the resetting rules are same, the update rules of search direction  can be viewed as different linear combinations of the history search direction and the current gradient. The calculation of $\beta^{(CG)}_k$ is based on $\nabla f(\mfx_k)$ and $\nabla f(\mfx_{k-1})$, while $\beta_k$ and $\gamma_k$ in SDC depend on the restarting criterion. Moreover, as mentioned before, the update rule of $\mfu_k$ with SDC involves an extra term $\gamma_k\frac{\|\mfu_k\|}{\|\nabla f(\mfx_k)\|} \nabla f(\mfx_k)$, which leads to a different combination rule. 

\section{SDC from an ODE perspective}
\label{sec:ode}
In this section, we consider the unconstrained smooth convex optimization problem \eqref{problem} with a unique minimizer $\mfx^*$. Namely, it is assumed that $h=0$, $f\in\mcF_L$ and $f$ is bounded from below. Moreover, we assume that no restarting criterion is triggered in Algorithm \ref{FISC} and the step size $s_k$ is fixed to be $s$. 
\subsection{SDC in continuous time}
By rescaling $\mfv_k=\sqrt{s}\mfu_k$, we can write the update rule of $\mfu_{k+1}$ and $\mfx_{k+1}$ given by \eqref{SDC-upd} and \eqref{x_upd} as follows:
\begin{equation}
\label{SDC_1}
\left\{
\begin{aligned}
&\frac{\mfv_{k+1}-\mfv_k}{\sqrt{s}}=-\frac{\beta_k}{\sqrt{s}}\mfv_k-\frac{\gamma_k}{\sqrt{s}}\frac{\|\mfv_k\|}{\|\nabla f(\mfx_k)\|}\nabla f(\mfx_k)-\nabla f(\mfx_k),\\
&\frac{\mfx_{k+1}-\mfx_k}{\sqrt{s}}=\mfv_{k+1}.
\end{aligned}
\right.
\end{equation}
Taking the limit $s\to 0$ in \eqref{SDC_1} and neglecting higher order terms, we directly have 
\begin{equation}
\label{SDC-ODE-1st}
\left\{
\begin{aligned}
&\dot \mfv=-\nabla f(\mfx)-\beta(t)\mfv+\gamma(t)\frac{||\mfv||}{||\nabla f(\mfx)||}\nabla f(\mfx),\\
&\dot \mfx=\mfv,\\
\end{aligned}
\right.
\end{equation}
where $\beta(t),\gamma(t): \mbR^+\to\mbR^+$ can be viewed as rescaled $\beta_k,\gamma_k$ in continuous time. Specifically, for FIRE, $\beta(t)$ and $\gamma(t)$ have the following expressions:
\begin{equation}
\label{bg_SDC}
\beta(t)=\gamma(t)=c_1e^{-c_2 t},
\end{equation}
where $c_1, c_2>0$ are constants. 

We can rewrite \eqref{SDC-ODE-1st} into a second-order ODE:
\begin{equation}
\label{SDC-ODE}
\ddot \mfx+\nabla f(\mfx)+\beta(t)\dot \mfx+\gamma(t)\frac{||\dot \mfx||}{||\nabla f(\mfx)||}\nabla f(\mfx)=0.
\tag*{(SDC-ODE)}
\end{equation}
Using the symplectic Euler scheme, the discretization of \eqref{SDC-ODE-1st} reads:
\begin{equation}
\label{dis_fo_ode}
\left\{
\begin{aligned}
&\mfv_{k+1}=\mfv_k-\sqrt{s}\nabla f(\mfx_k)-\sqrt{s}\beta(k\sqrt{s})\mfv_k+\sqrt{s}\gamma (k\sqrt{s})\frac{||\mfv_k||}{||\nabla f(\mfx_k)||}\nabla f(\mfx_k),\\
&\mfx_{k+1}=\mfx_k+\sqrt{s}\mfv_{k+1},\\
\end{aligned}
\right.
\end{equation}
where $\sqrt{s}$ is the step size.  By rescaling $\beta_k=\sqrt{s} \beta(k\sqrt{s})$, $\gamma_k=\sqrt{s} \gamma(k\sqrt{s})$ and $\mfu_k=\frac{1}{\sqrt{s}} \mfv_k$, \eqref{dis_fo_ode} is equivalent to the update rule \eqref{SDC_1}. In other word, we use \ref{SDC-ODE} to model these first-order methods with SDC. 

\subsection{FISC-ODE with a $\mcO(1/t^2)$ convergence rate}
The Lyapunov function (energy functional) is a powerful tool to analyze the convergence rate of ODE, as mentioned in \citep{avpoa}, \citep{alaom} and \citep{adefm}. But with $\beta(t),\gamma(t)$ specified by \eqref{bg_SDC}, \ref{SDC-ODE} is hard to be directly analyzed using Lyapunov's methods. We hope to choose proper $\beta_k$ and $\gamma_k$ to ensure that \ref{SDC-ODE} have certain good properties in Lyapunov analysis. Consider the following Lyapunov function for \ref{SDC-ODE}:
\begin{equation}
\label{lya_SDC}
\mcE(t)=\frac{\mu(t)}{2}\|\dot \mfx\|^2+\frac{1}{2}\|\mfx-\mfx^*+\phi(t)\dot \mfx\|^2+\zeta(t)(f(\mfx)-f(\mfx^*)).
\end{equation}
where $\mu(t), \phi(t)$ and $\zeta(t)$ are mappings $\mbR^+\to \mbR^+$ and $\mfx^*$ is the unique minimizer of $f$. The structure of  \eqref{lya_SDC} is motivated by the Lyapunov function in the works of \citet{avpoa} and \citet{drkda}. The Lyapunov function in \citep{avpoa} involves terms $\|\mfx(t)-\mfx^*+\phi(t)\dot \mfx(t)\|^2$ and $(f(\mfx(t))-f(\mfx^*))$ and \cite{drkda} introduces an additional term $\|\dot \mfx(t)\|^2$. 

We consider a specific selection of $ \beta(t), \gamma(t), \mu(t), \phi(t)$ and $\zeta(t)$:
\begin{equation}
\label{bgmpp}
\beta(t)=\frac{r-3}{t},\quad \gamma(t)=\frac{r}{t},\quad \mu(t)=\frac{(r-3)t^2}{2(r-1)^2}, \quad \phi(t)=\frac{t}{r-1}, \quad \zeta(t)=\frac{t^2}{2(r-1)},
\end{equation}
where $r\geq 3$ is a parameter. This renders our proposed \ref{FISC-ODE}:
\begin{equation}
\label{FISC-ODE}
\ddot \mfx+\frac{r}{t}\dot \mfx+\nabla f(\mfx)+\frac{r-3}{t}\frac{\|\dot \mfx\|}{\|\nabla f(\mfx)\|}\nabla f(\mfx)=0. \tag*{(FISC-ODE)}
\end{equation}
For the Lyapunov function of \ref{FISC-ODE} , we have the following lemma. 

\begin{lem}
\label{thm:lya}
With $\mu(t), \phi(t)$ and $\zeta(t)$ specified in \eqref{bgmpp}, the Lyapunov function $\mcE(t)$ satisfies $\dot \mcE(t)\leq0$.
\end{lem}
\begin{proof}
For simplicity, let $\omega=1/(r-1)$. Then $(r-3)/(r-1)^2=\omega-2\omega^2$, $r=(\omega+1)/\omega$. We can rewrite \ref{FISC-ODE} as:
\begin{equation}\label{omegaddx}
\omega t\ddot \mfx=-(1+\omega)\dot \mfx-\omega t \nabla  f(\mfx)-(1-2\omega)\frac{\|\dot \mfx\|}{\|\nabla f(\mfx)\|}\nabla f(\mfx).
\end{equation}
The convexity of $f$ yields 
\begin{equation}\label{cvx_f}
\la\mfx-\mfx^*, \nabla  f(\mfx)\ra\geq\la\mfx-\mfx^*, \nabla  f(\mfx)\ra -f(\mfx)+f(\mfx^*)\geq0.
\end{equation}
The Lyapunov function \eqref{lya_SDC} with $\mu(t), \phi(t)$ and $\zeta(t)$ specified in \eqref{bgmpp} writes
\begin{equation}
\label{lya_cont}
\mcE(t)=\frac{(\omega-2\omega^2)t^2}{4}\|\dot \mfx\|^2+\frac{1}{2}\left\|\mfx-\mfx^*+\omega t\dot \mfx\right\|^2+\frac{\omega t^2}{2}(f(\mfx)-f(\mfx^*)).
\end{equation}
Hence, we obtain
$$
\begin{aligned}
2\dot \mcE(t)=&(1-2\omega)t\la \dot \mfx, \omega t\ddot \mfx\ra+(\omega-2\omega^2)t\|\dot \mfx\|^2+2\la \mfx-\mfx^*+\omega t\dot \mfx, \dot \mfx+\omega\dot \mfx+\omega t\ddot \mfx\ra\\
&+\omega t^2 \la \dot \mfx, \nabla f(\mfx)\ra+2\omega t (f(\mfx)-f(\mfx^*))\\
=&-(1-2\omega)t\lp(1+\omega)\|\dot \mfx\|^2+\omega t\la \dot \mfx, \nabla  f(\mfx)\ra+(1-2\omega)\frac{\|\dot \mfx\|}{\|\nabla f(\mfx)\|} \la \dot \mfx, \nabla  f(\mfx)\ra\rp\\
&+(\omega-2\omega^2)t\|\dot \mfx\|^2-2(1-2\omega)\frac{\|\dot \mfx\|}{\|\nabla f(\mfx)\|}\la \mfx-\mfx^*+\omega t\dot \mfx, \nabla  f(\mfx)\ra\\
&-2\omega t\la \mfx-\mfx^*+\omega t\dot \mfx, \nabla  f(\mfx)\ra+\omega t^2 \la \dot \mfx, \nabla f(\mfx)\ra+2\omega t (f(\mfx)-f(\mfx^*))\\
=&-(1-2\omega)t\lp\|\dot\mfx\|^2+\frac{\|\dot \mfx\|}{\|\nabla f(\mfx)\|} \la \dot \mfx, \nabla  f(\mfx)\ra\rp\\
&-2(1-2\omega)\frac{\|\dot \mfx\|}{\|\nabla f(\mfx)\|}\la \mfx-\mfx^*, \nabla  f(\mfx)\ra\\
&-2\omega t\lp \la\mfx-\mfx^*, \nabla  f(\mfx)\ra -f(\mfx)+f(\mfx^*)\rp\leq0,
\end{aligned}
$$
where the second equality is due to \eqref{omegaddx} and the last inequality takes \eqref{cvx_f}.
\end{proof}

Based on Lemma \ref{thm:lya}, we have the following convergence rate of \ref{FISC-ODE}.

\begin{theorem}[The $\mcO(t^{-2})$ convergence rate of FISC-ODE]
\label{ot-2}
For any $r\geq3$, let $\mfx(t)$ be the solution to \ref{FISC-ODE} with initial conditions $\mfx(0)=\mfx_0$ and $\dot \mfx(0)=0$. Then, for $t>0$, we have
$$
f(\mfx(t))-f(\mfx^*)\leq\frac{(r-1)\|\mfx_0-\mfx^*\|^2}{t^2}.
$$
\end{theorem}
\begin{proof}
From Lemma \ref{thm:lya}, $\mcE(t)$ is non-increasing and by \eqref{lya_SDC}
$$
\mcE(t)\geq \frac{t^2(f(\mfx(t))-f(\mfx^*))}{2(r-1)}
$$
Hence, we obtain
$$
f(\mfx(t))-f(\mfx^*)\leq\frac{2(r-1)\mcE(t)}{t^2}\leq\frac{2(r-1)\mcE(0)}{t^2}=\frac{(r-1)\|\mfx_0-\mfx^*\|^2}{t^2}=O(t^{-2}),
$$
which completes the proof.
\end{proof}

Now, rewriting \ref{FISC-ODE} into a first-order ODE system and discretizing it with the symplectic Euler scheme, we can directly recover the update rule of FISC \eqref{FISC-upd} with $l_k=k$. We can also discretize \ref{FISC-ODE} with techniques analogous to the Nesterov's accelerated method, and then the update rule of FISC-ns \eqref{FISC-ns-upd} is recovered.  

\subsection{Comparison with other first-order methods with ODE interpretations}

If we take $r=3$, then \ref{FISC-ODE} turns to be
\begin{equation}
\label{Nesterov-ODE}
\ddot \mfx+\frac{3}{t}\dot \mfx+\nabla f(\mfx)=0. \tag*{(Nesterov-ODE)}
\end{equation}
\citet{adefm} used this ODE for modeling the Nesterov's accelerated method.

Dropping the term $(r-3)\|\dot \mfx\|\nabla f(\mfx)/(t\|\nabla f(\mfx)\|)$, \ref{FISC-ODE} becomes:
\begin{equation}
\label{HF-ns-ODE}
\ddot \mfx+\frac{r}{t}\dot \mfx+\nabla f(\mfx)=0, \tag*{(HF-ns-ODE)}
\end{equation}
which is the high friction version of \ref{Nesterov-ODE} in \citep{adefm} with $r\geq 3$. 

Under the special case $r=3$, the coefficient of the term $\frac{||\mfu_k||}{||\nabla f(\mfx_k)||}\nabla f(\mfx_k)$ in \eqref{FISC-ns-upd} turns to be $0$. If no restarting criterion is met and the step size is fixed, FISC-ns becomes the Nesterov's accelerated method. With restarts and a fixed step size, FISC-ns recovers the Nesterov's accelerated method with gradient restarting \citep{arfag}. Therefore, we can view FISC-ns as an extension of the restarting Nesterov's accelerated method. Furthermore,  numerical experiments indicate that a proper choice of $r$  leads to extra acceleration in the Nesterov's accelerated method.

We also observe that the ODE modeling the heavy-ball method is given by:
\begin{equation}
\label{HB-ODE}
\ddot \mfx+\beta\dot \mfx+\nabla f(\mfx)=0, \tag*{(HB-ODE)}
\end{equation}
where $\beta$ is a constant. The convergence rate of \ref{HB-ODE} is an open problem for the general convex $f$. 

In summary, \ref{Nesterov-ODE}, \ref{HF-ns-ODE} and \ref{HB-ODE} can be viewed as specific examples of \ref{SDC-ODE} with different choices of $\beta(t)$ and $\gamma(t)$. 

\section{Convergence analysis}
\label{sec:converge}
In this section, we analyze the global convergence of methods with SDC for general unconstrained smooth optimization problems and the convergence of FISC-PM for composite optimization problems. In both cases, we assume that the target function $f$ is bounded from below. 
\subsection{The global convergence of methods with SDC}
\label{ssec:glo}
In this subsection, we show the global convergence of methods with SDC and explain why we use \eqref{rst_crit} as our restarting criterion. We consider the case where the objective function is smooth, i.e., $h=0$ in \eqref{problem}. Define the level set
$$
\mcL=\{\mfx\in\mbR^n: f(\mfx)\leq f(\mfx_0)\}.
$$
Let $\bar \mcL$ be the collection of $\mfx\in\mbR^n$ whose distance to $\mcL$ is at most $\mu d_{max}$, where $d_{max}=\sup_k\|\nabla f(\mfx_k)\|$ and $\mu$ is a parameter. $f$ is assumed to be $L$-smooth on $\bar \mcL$. We begin with the following lemma:
\begin{lem}
\label{prop:stp}
Suppose that $f$ is differentiable. $\mfu_{k+1}$ is updated by \eqref{SDC-upd} or \eqref{SDC-rst} depending on the restarting criterion using $\varphi_k$, and $\mfx_{k+1}$ is updated by \eqref{x_upd}. $\beta_k$ and $\gamma_k$ satisfy \eqref{bg_cond}, and the step size is obtained by the nonmonotone line search. Then, for any integer $k\geq0$, we have 
\begin{equation}
\label{dir_ass1}
\la \mfu_{k+1}, -\nabla f(\mfx_k)\ra\geq \|\nabla f(\mfx_k)\|^2.
\end{equation}
\end{lem}

\begin{proof}
If $\varphi_k\geq0$, then we update $\mfu_{k+1}$ by \eqref{SDC-upd}. Hence,
$$
\begin{aligned}
&\la \mfu_{k+1}, -\nabla f(\mfx_k)\ra \\
= &(1-\beta_k)\la \mfu_k, -\nabla f(\mfx_k)\ra + \gamma_k \|\mfu_k\|\|\nabla f(\mfx_k)\|+\|\nabla f(\mfx_k)\|^2\geq \|\nabla f(\mfx_k)\|^2.
\end{aligned}
$$
If $\varphi_k<0$, we reset $\mfu_{k+1}=-\nabla f(\mfx_k)$ and $\la \mfu_{k+1}, -\nabla f(\mfx_k)\ra=\|\nabla f(\mfx_k)\|^2$. 
\end{proof}

Let $\theta_k$ be the angle between the search direction $\mfu_{k+1}$ and the negative gradient direction $-\nabla f(\mfx_k)$, i.e.,
$$
\theta_k=\arccos\frac{\la \mfu_{k+1}, -\nabla f(\mfx_k)\ra}{\|\mfu_{k+1}\|\|\nabla f(\mfx_k)\|}.
$$
According to \eqref{prop:stp}, we have a lower bound for $\cos\theta_k$:
\begin{equation}
\label{theta_bnd}
\cos\theta_k=\frac{\la \mfu_{k+1}, -\nabla f(\mfx_k)\ra}{\|\mfu_{k+1}\|\|\nabla f(\mfx_k)\|}\geq\frac{\|\nabla f(\mfx_k)\|}{\|\mfu_{k+1}\|}.
\end{equation}
Hence, $\la \mfu_{k+1}, -\nabla f(\mfx_k)\ra>0$ for each $k$. From our assumption that $f$ is bounded from below, there exists $s_k$ satisfying the Armijo conditions \eqref{armijo} or the nonmonotone Armijo conditions \eqref{equ:nonbls}, according to Lemma 1.1 in \citep{anlst}.

We add two restarting criteria:
\begin{gather}
\label{non_dec_grad}
d_f\|\nabla f(\mfx_{k})\|\geq \|\nabla f(\mfx_{k-1})\|,\\
\label{at_most_k}
n_k\leq K,
\end{gather}
where $K\in\mathbb{N}$, $d_f>1$ and $n_k$ is the number of iterations since the last restart. Namely, we restart our system if at least one of the criteria \eqref{rst_crit}, \eqref{non_dec_grad} and \eqref{at_most_k} is violated. If we set $d_f$ and $K$ large enough in practice, criteria \eqref{non_dec_grad} and \eqref{at_most_k} will seldom be violated. Equipped restarting criteria \eqref{non_dec_grad} and \eqref{at_most_k}, the system will restart at least once in $K$ consecutive iterations and $\|\nabla f(\mfx_{k})\|$ will not drop too rapidly. We then introduce the following lemma.

\begin{lem}
\label{upper_bound}
Suppose that the conditions of Lemma \ref{prop:stp} are satisfied. $K'<K$ is an integer and both \eqref{rst_crit} and \eqref{non_dec_grad} hold for $0\leq k\leq K'$. Then,  $\frac{\|\mfu_{k+1}\|}{\|\nabla f(\mfx_{k})\|}$ is upper bounded for all $0\leq k\leq K'$.
\end{lem}
\begin{proof}
Let $\lambda_k=\frac{\|\mfu_{k+1}\|}{\|\nabla f(\mfx_{k})\|}, \xi_k=\frac{\|\mfu_{k}\|}{\|\nabla f(\mfx_{k})\|}$. Specifically, $\lambda_{0}=1$. Based on \eqref{non_dec_grad}, we have $\xi_k\leq d_f\lambda_{k-1}$. Hence,
$$
\begin{aligned}
\lambda_k^2=&\frac{1}{\|\nabla f(\mfx_k)\|^2}\left[(1-\beta_k)^2\|\mfu_k\|^2+\left(1+\gamma_k \frac{\|\mfu_k\|}{\|\nabla f(\mfx_k)\|}\right)^2\|\nabla f(\mfx_k)\|^2\right.\\
&\left.+2(1-\beta_k)\left(1+\gamma_k \frac{\|\mfu_k\|}{\|\nabla f(\mfx_k)\|}\right)\la\mfu_k,-\nabla f(\mfx_k)\ra\right]\\
\leq&(1-\beta_k)^2\xi_k^2+(1+\gamma_k \xi_k)^2+2(1-\beta_k)(1+\gamma_k\xi_k)\xi_k\\
\leq& (1-\beta_k)^2d_f^2\lambda_{k-1}^2+(1+\gamma_k d_f\lambda_{k-1})^2+2(1-\beta_k)(1+\gamma_k d_f\lambda_{k-1})d_f\lambda_{k-1}\\
\leq&d_f^2\lambda_{k-1}^2+2(1+d_f\lambda_{k-1})^2=4d_f^2\lambda_{k-1}^2+4d_f\lambda_{k-1}+2.
\end{aligned}
$$
Consider a sequence $\{\tilde \lambda_k\}_{k=0}$ satisfying $\tilde\lambda_k^2=4d_f^2\tilde\lambda_{k-1}^2+4d_f\tilde\lambda_{k-1}+2$ and $\tilde \lambda_{0}=1$. Because $d_f>1$, it is obvious that $\tilde \lambda_k$ is increasing with respect to $k$. Then,
$$
\lambda_k\leq\tilde\lambda_{k}\leq \tilde \lambda_{K'}\leq \tilde \lambda_{K}, \quad 0\leq k\leq K',
$$
which concluded the proof.
\end{proof}

Lemma \ref{prop:stp} and \ref{upper_bound} guarantee that the direction assumption in \citep{anlst} holds. Namely, there exist positive constants $c_1$ and $c_2$ such that
\begin{equation}
\label{dir_ass}
\la \mfu_{k+1}, \nabla f(\mfx_k)\ra\leq -c_1\|\nabla f(\mfx_k)\|^2, \quad \|\mfu_{k+1}\|\leq c_2 \|\nabla f(\mfx_k)\|.
\end{equation}

Consider the sequence $\{\mfx_k\}$ given by Algorithm \ref{FISC} with extra restarting criteria \eqref{non_dec_grad} and \eqref{at_most_k}. We further assume that the step size $s_k$ is attained by the nonmonotone line search. Note that $f(\mfx)$ is bounded from below, the direction assumption \eqref{dir_ass} holds and the step sizes satisfy the nonmonotone Armijo conditions. According to Theorem 2.2 in \citep{anlst}, we obtain
$$
\lim\limits_{k\to\infty}\inf\|\nabla f(\mfx_k)\|=0.
$$
Moreover, if $\eta_{max}<1$ ($\eta_{max}$ is a parameter for the nonmonotone line search), then we have
$$
\lim\limits_{k\to\infty}\|\nabla f(\mfx_k)\|=0,
$$
which indicates the global convergence of first-order methods with SDC.

\subsection{The $\mcO(1/k^2)$ convergence rate of FISC-PM}  We analyze the convergence of FISC-PM for the composite optimization problem \eqref{problem} with a unique minimizer $\mfx^*$. It is assumed that $f\in \mcF_L$ is bounded from below. We consider the case that the step size is fixed to be $0<s\leq 1/L$ and no restarts are used, i.e., the sequences $\{\mfx_k\}$ and $\{\mfy_k\}$ are merely updated by \eqref{SDC-PM-upd} and \eqref{yy_upd}. $\beta_k, \gamma_k$ are specified by \eqref{FISC-para} with $l_k=k$. We introduce the following discrete Lyapunov function $\mathcal{E}(k)$:

\begin{eqt}
\label{lya_dis}
\mathcal{E}(k) = &2\left\|\mfx_k - \mfx^* + \frac{k-1}{r-1}(\mfx_k - \mfx_{k-1})\right\|^2 + \frac{2(k+r-2)^2s}{r-1}(f(\mfx_k) - f(\mfx^*)) \\
&+ \frac{(r-3)(k-1)^2}{(r-1)^2}\|\mfx_k - \mfx_{k-1}\|^2.
\end{eqt}
The function $\mathcal{E}(k)$ can be viewed as the discrete version of \eqref{lya_cont} by multiplying $4$. We introduce a basic inequality in convex optimization:
\begin{lem}
Consider a convex function of the form $f(\mfx)=\varphi(\mfx)+h(\mfx)$, where $\varphi\in \mcF_L$ and $h$ is convex. For any $0<s\le 1/L$ and $\mfx,\mfy\in\mbR^n$, we have
\begin{equation}
\label{basic_ineq}
f(\mfy - sG_s(\mfy))\le f(\mfx) + G_s(\mfy)^T(\mfy - \mfx) - \frac{s}{2}\|G_s(\mfy)\|^2.
\end{equation}
\end{lem}
Based on the basic inequality \eqref{basic_ineq}, we give the following Lemma \ref{lemma:lya}.

\begin{lem}
\label{lemma:lya}
The discrete Lyapunov function $\mcE(k)$ given by \eqref{lya_dis} satisfies 
\begin{equation}
\label{estamate_lya}
\mcE(k)-\mcE(k-1)\leq\alpha(\phi_{k-1}-2)\|\Delta \mfx_{k-1}\|^2-\alpha\phi_k\|\Delta \mfx_k\|^2-\frac{2s}{r-1}(f(\mfx_{k-1})-f(\mfx^*)),
\end{equation}
where 
\begin{equation}
\label{notations}
\alpha=\frac{r-3}{r-1},\, \phi_k=2k+r-3,\, \Delta \mfx_k = \mfx_k-\mfx_{k-1}.
\end{equation}
\end{lem}
\begin{proof}
For simplicity, we denote 
$$
\mfr_k=\frac{\|\Delta \mfx_k\|}{\|G_s(\mfx_k)\|}G_s(\mfx_k), \,\,\xi_k=\frac{k+r-2}{r-1}, \,\,\nu_k = \frac{2(k+r-2)(k+r-4)}{r-1},
$$
and introduce two auxiliary variables $\mfz_k$ and $\mfw_k$ defined by
\begin{equation}
\label{zk_def}
\mfz_k  = \mfx_k + \frac{k-1}{r-1}\Delta \mfx_k, \,\,\mfw_k = \mfz_k + \mfz_{k-1}-\mfx_k-\mfx_{k-1}.
\end{equation}
We can also write $\mfz_{k-1}$ in the following way:
\begin{eqt}
\label{A3}
&\mfz_{k-1}=\mfx_{k-1}+\frac{k-2}{r-1}\Delta \mfx_{k-1}=\frac{k+r-2}{r-1}\left(\mfx_{k-1}+\frac{k-2}{k+r-2}\Delta \mfx_{k-1}\right)-\frac{k-1}{r-1}\mfx_{k-1}\\
=&\frac{k+r-2}{r-1}\left(\mfy_{k-1}+\frac{r-3}{k+r-2}\mfr_{k-1}\right)-\frac{k-1}{r-1}\mfx_{k-1}=\xi_k\mfy_{k-1}+\alpha \mfr_{k-1}-\frac{k-1}{r-1}\mfx_{k-1}.
\end{eqt}
The update rule \eqref{SDC-PM-upd} and \eqref{yy_upd} can be written as:
\begin{equation}
\label{A4}
\frac{k-2}{r-1}\Delta \mfx_{k-1}-\alpha \mfr_{k-1}=\xi_k\left(\Delta \mfx_k+sG_s(\mfy_{k-1})\right).
\end{equation}

Based on the equations \eqref{A3} and \eqref{A4}, we can write
\begin{eqt}
& \mfz_k - \mfz_{k-1} = \Delta \mfx_k+ \frac{k-1}{r-1} \Delta \mfx_k - \frac{k-2}{r-1} \Delta \mfx_{k-1}\\
=&\xi_k\Delta \mfx_k - \frac{k-2}{r-1} \Delta \mfx_{k-1}= -\alpha \mfr_{k-1} -\xi_ksG_s(\mfy_{k-1}).
\end{eqt}
\begin{eqt}
\label{A5}
&\mfz_k+\mfz_{k-1}=\mfz_k-\mfz_{k-1}+2\mfz_{k-1}\\
=&-\alpha \mfr_{k-1} - \xi_ksG_s(\mfy_{k-1})+2\xi_{k}\mfy_{k-1} - \frac{2(k-1)}{r-1}\mfx_{k-1} + 2\alpha \mfr_{k-1}\\
=&- \xi_ksG_s(\mfy_{k-1})+2\xi_{k}\mfy_{k-1} - \frac{2(k-1)}{r-1}\mfx_{k-1} + \alpha \mfr_{k-1}.
\end{eqt}
Using the equations \eqref{zk_def}, \eqref{A4} and the fact $\frac{k-1}{r-1}+\xi_k=\frac{2k+r-3}{r-1}=\frac{\phi_k}{r-1}$ yields
\begin{eqt}
\label{wk_p1}
&\mfw_k=\frac{k-1}{r-1}\Delta \mfx_k+ \frac{k-2}{r-1}\Delta \mfx_{k-1}=\frac{k-1}{r-1}\Delta \mfx_k+\alpha \mfr_{k-1}+\xi_k\left(\Delta \mfx_k+sG_s(\mfy_{k-1})\right)\\
=&\frac{\phi_k}{r-1}\Delta \mfx_k+\alpha \mfr_{k-1}+\xi_ksG_s(\mfy_{k-1}).
\end{eqt}

We now analyze the difference between $2\left\|\mfx_k - \mfx^* + \frac{k-1}{r-1}\Delta \mfx_k\right\|^2$ in $\mathcal{E}(k)$:

\begin{eqt}
\label{A7}
& 2\left\|\mfx_k - \mfx^* + \frac{k-1}{r-1}\Delta \mfx_k\right\|^2 - 2\left\|\mfx_{k-1} - \mfx^* + \frac{k-2}{r-1}\Delta \mfx_{k-1}\right\|^2\\
	=&  2\|\mfz_k - \mfx^*\|^2 - 2\|\mfz_{k-1} - \mfx^*\|^2 =  2(\mfz_k - \mfz_{k-1})^T(\mfz_k + \mfz_{k-1} - 2\mfx^*)\\
	=&  -2\xi_ksG_s(\mfy_{k-1})^T(\mfz_k + \mfz_{k-1} - 2\mfx^*) - 2\alpha \mfr_{k-1}^T(\mfz_k + \mfz_{k-1} - 2\mfx^*)\\
	=&  -2\xi_ksG_s(\mfy_{k-1})^T(\mfz_k + \mfz_{k-1} - 2\mfx^*)- 2\alpha \mfr_{k-1}^T\left(\mfx_k + \mfx_{k-1} - 2\mfx^*\right) - 2\alpha \mfr_{k-1}^T\mfw_k. \\
\end{eqt}
Then, the difference between $\frac{(r-3)(k-1)^2}{(r-1)^2}\|\Delta \mfx_k\|^2$ in $\mathcal{E}(k)$ is calculated by
\begin{eqt}
\label{A9}
& \frac{(r-3)(k-1)^2}{(r-1)^2}\|\Delta \mfx_k\|^2 - \frac{(r-3)(k-2)^2}{(r-1)^2}\|\Delta \mfx_{k-1}\|^2\\
	=&  (r-3)(\|\mfz_k - \mfx_k\|^2 - \|\mfz_{k-1} - \mfx_{k-1}\|^2)=(r-3)(\mfz_k - \mfz_{k-1} - \Delta \mfx_k)^T\mfw_k\\
	=& (r-3)\left(-\alpha \mfr_{k-1} - \xi_ksG_s(\mfy_{k-1}) - \Delta \mfx_k\right)^T\mfw_k.
\end{eqt}
By using \eqref{A7} and \eqref{A9}, we can split $\mathcal{E}(k) - \mathcal{E}(k-1)$ into three parts:

\begin{eqt}
\label{A10}
	& \mathcal{E}(k) - \mathcal{E}(k-1) \\
	=& -2\xi_kG_s(\mfy_{k-1})^T(\mfz_k + \mfz_{k-1} - 2\mfx^*) - 2\alpha \mfr_{k-1}^T\left(\mfx_k + \mfx_{k-1} - 2\mfx^*\right)\\
	&- 2\alpha\mfr_{k-1}^T\mfw_k-(r-3)\left(\alpha \mfr_{k-1}^T\mfw_k+\xi_ksG_s(\mfy_{k-1})^T\mfw_k +\Delta \mfx_k^T\mfw_k\right)\\
	& + \frac{2(k+r-2)^2s}{r-1}(f(\mfx_k) - f(\mfx^*)) - \frac{2(k+r-3)^2s}{r-1}(f(\mfx_{k-1}) - f(\mfx^*))\\
	=&-(r-3)(\mfr_{k-1} + \Delta \mfx_k)^T\mfw_k- 2\alpha \mfr_{k-1}^T\left(\mfx_k + \mfx_{k-1} - 2\mfx^*\right)\\
	&-2\xi_ksG_s(\mfy_{k-1})^T(\mfz_k + \mfz_{k-1} - 2\mfx^*)-(r-3)\xi_ksG_s(\mfy_{k-1})^T\mfw_k\\
	& + \frac{2(k+r-2)^2s}{r-1}(f(\mfx_k) - f(\mfx^*)) - \frac{2(k+r-3)^2s}{r-1}(f(\mfx_{k-1}) - f(\mfx^*)).\\
\end{eqt}
The quantities in the last three rows of \eqref{A10} are denoted as $L_1$, $L_2$ and $L_3$, respectively. From \eqref{wk_p1} and $r-3=\alpha(r-1)$, it follows that 
\begin{eqt}
\label{B17}
&L_1+4\alpha \mfr_{k-1}^T(\mfx_{k-1}-\mfx^*)= -(r-3)(\mfr_{k-1} + \Delta \mfx_k)^T\mfw_k-2\alpha \mfr_{k-1}^T\Delta \mfx_k\\
=& -(r-3)(\mfr_{k-1} + \Delta \mfx_k)^T\left(\frac{\phi_k}{r-1}\Delta \mfx_k+\alpha \mfr_{k-1}+\xi_ksG_s(\mfy_{k-1})\right)-2\alpha \mfr_{k-1}^T\Delta \mfx_k\\
=&-(r-3)\xi_k\mfr_{k-1}^TsG_s(\mfy_{k-1})-(r-3)\xi_k\Delta \mfx_k^TsG_s(\mfy_{k-1})\\
&-\alpha (\phi_k\|\Delta \mfx_k\|^2+(r-3)\|\mfr_{k-1}\|^2))-2(r-3)\xi_k\mfr_{k-1}^T\Delta \mfx_k\\
=&-\alpha \left(\phi_k\|\Delta \mfx_k\|^2+(r-3)\|\mfr_{k-1}\|^2+2(r-1)\xi_k\mfr_{k-1}^T(\Delta \mfx_k+sG_s(\mfy_{k-1}))\right)+\bar L_1,\\
\end{eqt}
where 
\begin{eqt}
\label{bar_L1}
\bar L_1 = &(r-3)\xi_k\mfr_{k-1}^TsG_s(\mfy_{k-1})-(r-3)\xi_k\Delta \mfx_k^TsG_s(\mfy_{k-1})\\
=&(r-3)\xi_ksG_s(\mfy_{k-1})^T(\mfr_{k-1}-\Delta \mfx_k).
\end{eqt}
Utilizing the equation \eqref{A4} and $\|\mfr_{k-1}\|=\|\Delta \mfx_{k-1}\|$, we obtain

\begin{eqt}
\label{L1_part}
&\phi_k\|\Delta \mfx_k\|^2+(r-3)\|\mfr_{k-1}\|^2+2(r-1)\xi_k\mfr_{k-1}^T(\Delta \mfx_k+sG_s(\mfy_{k-1}))\\
=&\phi_k\|\Delta \mfx_k\|^2+(r-3)\|\mfr_{k-1}\|^2+2\mfr_{k-1}^T((k-2)\Delta \mfx_{k-1}-(r-3)\mfr_{k-1})\\
=&\phi_k\|\Delta \mfx_k\|^2-(r-3)\|\mfr_{k-1}\|^2+2(k-2)\mfr_{k-1}^T\Delta \mfx_{k-1}\\
\geq&\phi_k\|\Delta \mfx_k\|^2-(2k-r-7)\|\mfr_{k-1}\|^2=\phi_k\|\Delta \mfx_k\|^2-(\phi_{k-1}-2)\|\Delta \mfx_{k-1}\|^2.\\
\end{eqt}
The last inequality even holds when $k=1$ because $\mfr_0=\Delta \mfx_0=0$. By setting $\mfy=\mfx_{k-1}$, $\mfx=\mfx^*$ in the basi inequality \eqref{basic_ineq}, we have 
\begin{eqt}
\label{inequ1}
&\frac{\|G_s(\mfx_{k-1})\|}{\|\Delta \mfx_{k}\|}\mfr_{k-1}^T(\mfx_{k-1}-\mfx^*)=G_s(\mfx_{k-1})^T(\mfx_{k-1}-\mfx^*)\\
\geq& f(\mfx_{k-1}-sG_s(\mfx_{k-1}))-f(\mfx^*)+\frac{s}{2}\|G_s(\mfx_{k-1})\|^2\geq0.
\end{eqt} Substituting inequalities \eqref{L1_part} and \eqref{inequ1} in \eqref{B17} yields
\begin{equation}
\label{L1_est}
L_1 \leq -\alpha\phi_k\|\Delta \mfx_k\|^2+\alpha(\phi_{k-1}-2)\|\Delta \mfx_{k-1}\|^2+\bar L_1.\\
\end{equation}

From the definition of $\mfw_k$ and the equation \eqref{A5}, we obtain
\begin{eqt}
&2(\mfz_k + \mfz_{k-1} - 2\mfx^*)+(r-3)\mfw_k\\
=&2(\mfz_k + \mfz_{k-1} - 2\mfx^*)+(r-3)((\mfz_k + \mfz_{k-1} - 2\mfx^*)-(\mfx_k+\mfx_{k-1}-2\mfx^*))\\
=&(r-1)\left(- \xi_ksG_s(\mfy_{k-1})+\xi_{k}2\mfy_{k-1} - \frac{2(k-1)}{r-1}\mfx_{k-1}+\alpha\mfr_{k-1}- 2\mfx^*\right)\\
&-2(r-3)(\mfx_{k-1}-\mfx^*)-(r-3)\Delta \mfx\\
=&(r-1)\left(-\xi_ksG_s(\mfy_{k-1})+2\xi_{k}\mfy_{k-1} - \frac{2(k-1)}{r-1}\mfx_{k-1}- 2\mfx^*\right)\\
&-2(r-3)(\mfx_{k-1}-\mfx^*)-(r-3)(\Delta \mfx-\mfr_{k-1}).\\
\end{eqt}
The above estimation implies
\begin{eqt}
\label{L2}
&L_2 = -\xi_ksG_s(\mfy_{k-1})^T(2(\mfz_k + \mfz_{k-1} - 2\mfx^*)+(r-3)\mfw_k)\\
=&-(r-1)\xi_ksG_s(\mfy_{k-1})^T\left(- \xi_ksG_s(\mfy_{k-1})+2\xi_{k}\mfy_{k-1} - \frac{2(k-1)}{r-1}\mfx_{k-1}- 2\mfx^*\right)\\
&+2(r-3)\xi_ksG_s(\mfy_{k-1})^T(\mfx_{k-1}-\mfx^*)+(r-3)\xi_ksG_s(\mfy_{k-1})^T(\Delta \mfx_k-\mfr_{k-1}).\\
\end{eqt}

Finally, we compute $L_3$. Note that $\mfx_k=\mfy_{k-1}-sG_s(\mfy_{k-1})$. Taking $\mfy=\mfy_{k-1}$, $\mfx=\mfx_k$ or $\mfx^*$ in the basic inequality \eqref{basic_ineq} gives
\begin{eqt}
\label{inequ2}
& f(\mfx_k)\le f(\mfx_{k-1}) + G_s(\mfy_{k-1})^T(\mfy_{k-1} - \mfx_{k-1}) - \frac{s}{2}\|G_s(\mfy_{k-1})\|^2,\\
& f(\mfx_k)\le f(\mfx^*) + G_s(\mfy_{k-1})^T(\mfy_{k-1} - \mfx^*) - \frac{s}{2}\|G_s(\mfy_{k-1})\|^2.
\end{eqt}
Based on the above inequalities, we observe that
\begin{eqt}
\label{L3_est}
&L_3+\frac{2s}{r-1}(f(\mfx_{k-1}) - f(\mfx^*))\\
=& \frac{2(k+r-2)^2s}{r-1}(f(\mfx_k) - f(\mfx^*)) - \frac{2(k+r-2)(k+r-4)s}{r-1}(f(\mfx_{k-1}) - f(\mfx^*))\\
=&4\xi_k s(f(\mfx_k)-f(\mfx^*))+\nu_k s(f(\mfx_k)-f(\mfx_{k-1}))\\
\leq &4\xi_k\left(sG_s(y_{k-1})^T(\mfy_{k-1}-\mfx^*)-\frac{s}{2}\|G_s(\mfy_{k-1})\|^2\right)\\
&+\nu_k\lp G_s(\mfy_{k-1})^T(\mfy_{k-1} - \mfx^*) - \frac{s}{2}\|G_s(\mfy_{k-1})\|^2\rp= \bar L_3.
\end{eqt}
Note that $4\xi_k+\nu_k=\frac{2(k+r-2)^2}{r-1}=\frac{(r-1)\xi_k^2}{8}$. $\bar L_3$ can be rewritten into

\begin{eqt}
\label{bar_L3}
\bar L_3=& (4\xi_k+\nu_k)s\left(G_s(\mfy_{k-1})^T\mfy_{k-1}-\frac{s}{2}\|G_s(\mfy_{k-1})\|^2\right)-sG_s(\mfy_{k-1})^T(4\xi_k\mfx^*+\nu _k\mfx_{k-1})\\
=&2(r-1)\xi_k^2s\left(G_s(\mfy_{k-1})^T\mfy_{k-1}-\frac{s}{2}\|G_s(\mfy_{k-1})\|^2\right)\\
    &-sG_s(\mfy_{k-1})^T\lp2((r-1)-(r-3))\xi_k\mfx^*+2((k-1)-(r-3))\xi_k\mfx_{k-1}\rp\\
 =&(r-1)\xi_ksG_s(\mfy_{k-1})^T\left(2\xi_k\mfy_{k-1} - \frac{2(k-1)}{r-1}\mfx_{k-1} - 2\mfx^*-\xi_ksG_s(\mfy_{k-1})\right)\\
 &  + 2(r-3)\xi_ksG_s(\mfy_{k-1})^T(\mfx^* - \mfx_{k-1}).
\end{eqt}
Together with the equations \eqref{bar_L1} and \eqref{L2}, we have 
\begin{equation}
\label{l123}
\bar L_1+L_2+\bar L_3=0.
\end{equation}
Therefore, substituting \eqref{L1_est}, \eqref{L3_est} and \eqref{l123} in \eqref{A10} renders \eqref{estamate_lya}.
\end{proof} 

Based on Lemma \ref{lemma:lya}, we have the following estimation of $\mcE(k)$.

\begin{lem}[Discrete Lyapunov analysis of FISC-PM]
\label{prop:convergence}
The Lyapunov function $\mathcal{E}(k)$ defined in \eqref{lya_dis} satisfies
\begin{equation}
\mathcal{E}(k)\leq\mathcal{E}(0) -\frac{2s}{r-1}(f(\mfx_{0})-f(\mfx^*)).
\end{equation}
\end{lem}
\begin{proof}
Note that $\Delta \mfx_{0}=\mfx_0-\mfx_{-1}=0$. Summing \eqref{estamate_lya} for $l=1$ to $k$ yields

\begin{eqt}
\mcE(k)-\mcE(0)\leq&\alpha\sum_{l=1}^k\left((\phi_k-2)\|\Delta \mfx_{l-1}\|^2-\phi_k\|\Delta \mfx_l\|^2\right)-\frac{2s}{r-1}\sum_{l=1}^k(f(\mfx_{l-1})-f(\mfx^*))\\
\leq&\alpha\left(-\phi_k\|\Delta \mfx_k\|^2-2\sum_{l=2}^{k-1}\|\Delta \mfx_l\|^2\right)-\frac{2s}{r-1}(f(\mfx_{0})-f(\mfx^*))\\
\leq& -\frac{2s}{r-1}(f(\mfx_{0})-f(\mfx^*)).\\
\end{eqt}
\end{proof}

Theorem \ref{ot-2} tells that FISC-ODE has the $\mcO(t^{-2})$ convergence rate and the following theorem  is a discretized analog of Theorem \ref{ot-2}.

\begin{theorem}[The $\mcO(k^{-2})$ convergence rate of FISC-PM]
\label{prop:Ok2}
Let $\{\mfx_k\}$ be a sequence given by \eqref{SDC-PM-upd} and \eqref{yy_upd}. The step size is fixed as $0<s\leq1/L$ and $\beta_k, \gamma_k$ are specified by \eqref{FISC-para} with $l_k=k$. Then, we have
$$
f(\mfx_k)-f(\mfx^*)\leq\frac{(r-1)C_0}{2(k+r-2)^2s}=\mcO(k^{-2}),
$$
where 
$$
C_0=\mcE(0)-\frac{2s}{r-1}(f(\mfx_{0})-f(\mfx^*))=2\|\mfx_0-\mfx^*\|^2+(r-3)s\left(f(\mfx_0)-f(\mfx^*)\right).
$$
\end{theorem}
\begin{proof}
By Lemma \ref{prop:convergence}, the sequence of $\{\mfx_k\}$ given by FISC-PM satisfies
$$
f(\mfx_k)-f(\mfx^*)\leq\frac{r-1}{2(k+r-2)^2s}\mathcal{E}(k)\leq\frac{r-1}{2(k+r-2)^2s}C_0=\mcO(k^{-2}),
$$
which completes the proof.
\end{proof}

Note that FISC-ns is FISC-PM with $h=0$. Hence, we also prove the $\mcO(k^{-2})$ convergence rate of FISC-ns for smooth convex optimization problems. 

\section{Numerical Experiments}
\label{sec:num}
\subsection{The Lagrangian form of Lasso}
\label{lasso}
We compare FIRE, FISC and other optimization solvers on the following problem:
$$
\min\limits_{\mfx\in\mbR^n}\frac{1}{2}\|A\mfx-b\|^2+\lambda \|\mfx\|_1.
$$
Here we have $\psi(\mfx)=\frac{1}{2}\|A\mfx-b\|^2, \, h(\mfx)=\lambda \|\mfx\|_1$, where $A\in \mbR^{m\times n}, b\in \mbR^m, \lambda>0$. The proximal mapping is computed as
\begin{equation}\label{l1_prox}
\lp\mathrm{prox}_{h}^{s}(\mfx)\rp_i=\text{sign}(\mfx_i)\max\{|\mfx_i|-\lambda s,0\}.
\end{equation}
In our numerical experiment, $\lambda$ varies from different test cases and it is around $8\times 10^{-3}$. 

\subsubsection{Algorithm details and the implementation}
We describe the implementation details of our method and of the state-of-the-art algorithms used in our numerical comparison. The solvers used for comparison include SNF \citep{asnmw}, ASSN \citep{arssn}, FPC-AS \citep{afafs} and SpaRSA \citep{srbsa}. We give an overview of the tested algorithms:
\begin{itemize}
\item SNF is a semi-smooth Newton type method which uses the filter strategy. 
\item SNF(aCG) is the SNF solver with an adaptive parameter strategy in the CG method for solving the Newton equation.
\item ASSN is an adaptive semi-smooth Newton method.
\item FPC-AS is a first-order method that uses a fixed-point iteration under Barzilai-Borwein (BB) steps \citep{BBstep} and the continuation strategy.
\item SpaRSA, which resembles FPC-AS, is also a first-order method using BB steps and the continuation strategy.
\item F-PG(M)/FS-PG(M)($r$) is the FIRE/FISC algorithm using the proximal gradient (the proximal mapping) with the continuation strategy. The step size is obtained from the nonmonotone line search with the BB step as the initial guess. The number in the bracket is the parameter $r$ for FISC-PG(M).  
\end{itemize}
The continuation strategy in F-PG(M)/FS-PG(M) is same as in \citep{afafs}. Note that FISC-PM with $r=3$ recovers FISTA. We take same parameters for ASSN, FPC-AS, SpaRSA and SNF as in \citep{asnmw}. 

\subsubsection{The numerical comparison}
We use test problems from \citep{asnmw}, which are constructed as follows. Firstly, we randomly generate a sparse solution $\bar \mfx \in \mbR^n$ with $k$ nonzero entries, where $n = 512^2 = 262144$ and $k = [n/40] = 5553$. The $k$ different indices are uniformly chosen from $\{1, 2, . . . , n\}$ and the magnitude of each nonzero element is set by $\bar x_i =c_1(i)10^{dc_2(i)/20} $, where $c_1(i)$ is randomly chosen from $\{-1, 1\}$ with probability 1/2, respectively, $c_2 (i)$ is uniformly distributed in $[0, 1]$ and $d$ is a dynamic range which can influence the efficiency of the solvers. Then we choose $m = n/8 = 32768$ random cosine measurements, i.e., $A\bar \mfx = (dct(\bar \mfx))_J$, where J contains $m$ different indices randomly chosen from $\{1, 2, . . . , n\}$ and $dct$ is the discrete cosine transform. Finally, we construct the input data by $b = A\bar \mfx +w$, where $w$ is an isotropic Gaussian noise with a standard deviation $\bar \sigma = 0.1$.

To compare fairly, we set a uniform stopping criterion. For a certain tolerance $\epsilon$, we obtain a solution $\mfx_{newt}$ using ASSN \citep{arssn} such that $\|s G_s(\mfx_{newt})\|\leq\epsilon$. Then, we terminate all methods by the relative criterion
$$
\frac { f \left( \mfx ^ { k } \right) - f \left( \mfx ^ { * } \right) } { \max \left\{ |f \left( \mfx ^ { * } \right)| , 1 \right\} } \leq \frac { f \left( \mfx _ { n e w t } \right) - f \left( \mfx ^ { * } \right) } { \max \left\{ |f \left( \mfx ^ { * } \right)| , 1 \right\} },
$$
where $f(\mfx)$ is the objective function and $\mfx^*$ is a highly accurate solution using ASSN \citep{arssn} under the criterion $\|s G_s(\mfx^*)\|\leq10^{-13}$.

We solve the test problems under different tolerances $\epsilon \in \{ 10 ^ { - 0 } , 10 ^ { - 1 } , 10 ^ { - 2 } , 10 ^ { - 4 } ,$ $ 10 ^ { - 6 } \}$ and dynamic ranges $d \in \{ 20,40,60,80 \}$. Since the evaluations of $dct$ dominate the overall computation, we mainly use the total numbers of $A$-calls and $A^T$-calls $N_A$ to compare the efficiency of different solvers. Tables \ref{20db}-\ref{80db} show the average numbers of $N_A$ and CPU time over $10$ independent trials. 

\begin{table}[!htp]
\caption{Total number of $A$-calls and $A^T$-calls $N_A$ and CPU time (in seconds) averaged over $10$ independent runs with dynamic range $20 dB$}\label{20db}
\setlength{\tabcolsep}{0.9mm}{
\begin{tabular}{lllllllllll}
\hline
Method&\multicolumn{2}{l}{$\epsilon:10^0$}&\multicolumn{2}{l}{$\epsilon:10^{-1}$}&\multicolumn{2}{l}{$\epsilon:10^{-2}$}&\multicolumn{2}{l}{$\epsilon:10^{-4}$}&\multicolumn{2}{l}{$\epsilon:10^{-6}$}\\
\cline{2-11}
&Time&$N_A$&Time&$N_A$&Time&$N_A$&Time&$N_A$&Time&$N_A$\\
\midrule[1pt]
SNF&$1.09$&$84.6$&$2.63$&$205.0$&$3.20$&$254.2$&$3.85$&$307.0$&$4.59$&$373.0$\\
SNF(aCG)&$1.11$&$84.6$&$2.62$&$205.0$&$3.24$&$254.2$&$4.13$&$331.2$&$6.62$&$486.2$\\
ASSN&$1.13$&$89.8$&$1.82$&$145.0$&$2.10$&$\mathbf{173.0}$&$\mathbf{2.97}$&$\mathbf{246.4}$&$\mathbf{3.55}$&$\mathbf{298.2}$\\
FPC-AS&$1.45$&$109.8$&$5.08$&$366.0$&$6.88$&$510.4$&$9.56$&$719.4$&$9.90$&$740.8$\\
SpaRSA&$4.92$&$517.2$&$4.84$&$519.2$&$5.12$&$539.8$&$5.86$&$627.0$&$6.61$&$705.8$\\
F-PG     &$2.14$ &$190.4$ &$3.21$ &$291.2$ &$4.25$ &$376.8$ &$6.79$ &$600.8$ &$9.05$ &$801.8$ \\
FS-PG(3) &$0.81$ &$71.2$ &$1.34$ &$\mathbf{119.4}$ &$\mathbf{1.93}$ &$175.6$ &$3.24$ &$283.8$ &$4.48$ &$394.8$ \\
FS-PG(5) &$\mathbf{0.70}$ &$\mathbf{64.4}$ &$\mathbf{1.32}$ &$121.2$ &$2.07$ &$182.0$ &$3.24$ &$286.6$ &$4.39$ &$390.2$ \\
F-PM     &$0.95$ &$81.8$ &$1.54$ &$140.0$ &$2.11$ &$180.4$ &$3.91$ &$338.8$ &$5.14$ &$464.2$ \\
FS-PM(3) &$1.12$ &$97.0$ &$1.97$ &$168.0$ &$3.51$ &$298.6$ &$6.71$ &$596.0$ &$9.39$ &$817.0$ \\
FS-PM(5) &$0.98$ &$87.4$ &$1.68$ &$141.4$ &$2.65$ &$227.0$ &$6.36$ &$560.2$ &$8.08$ &$702.2$ \\
\hline
\end{tabular}}
\end{table}

\begin{table}[!htp]
\caption{Total number of $A$-calls and $A^T$-calls $N_A$ and CPU time (in seconds) averaged over $10$ independent runs with dynamic range $40 dB$}
\setlength{\tabcolsep}{0.9mm}{
\begin{tabular}{lllllllllll}
\hline
Method&\multicolumn{2}{l}{$\epsilon:10^0$}&\multicolumn{2}{l}{$\epsilon:10^{-1}$}&\multicolumn{2}{l}{$\epsilon:10^{-2}$}&\multicolumn{2}{l}{$\epsilon:10^{-4}$}&\multicolumn{2}{l}{$\epsilon:10^{-6}$}\\
\cline{2-11}
&Time&$N_A$&Time&$N_A$&Time&$N_A$&Time&$N_A$&Time&$N_A$\\
\midrule[1pt]
SNF&$2.06$&$158.2$&$5.01$&$380.8$&$6.19$&$483.2$&$6.69$&$525.0$&$7.16$&$566.8$\\
SNF(aCG)&$2.08$&$158.2$&$4.97$&$380.8$&$6.16$&$483.2$&$7.07$&$553.6$&$7.30$&$580.0$\\
ASSN&$2.28$&$182.2$&$3.53$&$285.4$&$4.10$&$338.6$&$4.97$&$\textbf{407.0}$&$\mathbf{5.56}$&$\textbf{459.2}$\\
FPC-AS&$2.12$&$158.0$&$5.34$&$399.2$&$7.72$&$578.4$&$9.62$&$720.2$&$10.41$&$774.8$\\
SpaRSA&$5.05$&$523.4$&$5.07$&$530.0$&$5.56$&$588.2$&$6.38$&$671.6$&$7.28$&$755.8$\\
F-PG     &$4.28$ &$378.0$ &$5.76$ &$522.4$ &$7.28$ &$642.8$ &$9.28$ &$813.6$ &$11.05$ &$990.0$ \\
FS-PG(3) &$1.71$ &$153.6$ &$3.05$ &$276.4$ &$3.94$ &$354.6$ &$4.89$ &$439.6$ &$6.37$ &$567.2$ \\
FS-PG(5) &$\mathbf{1.62}$ &$\mathbf{143.6}$ &$\mathbf{2.72}$ &$\mathbf{245.6}$ &$\mathbf{3.46}$ &$\mathbf{317.6}$ &$\mathbf{4.41}$ &$415.6$ &$5.68$ &$518.0$ \\
F-PM     &$2.02$ &$171.2$ &$2.68$ &$244.0$ &$3.94$ &$347.4$ &$5.34$ &$480.8$ &$7.09$ &$626.2$ \\
FS-PM(3) &$2.11$ &$184.2$ &$3.14$ &$279.8$ &$4.68$ &$424.0$ &$7.29$ &$648.2$ &$10.11$ &$903.4$ \\
FS-PM(5) &$2.17$ &$191.2$ &$3.50$ &$308.8$ &$4.54$ &$401.4$ &$6.07$ &$537.0$ &$8.25$ &$716.8$ \\
\hline
\end{tabular}}
\end{table}
\begin{table}[!htp]
\caption{Total number of $A$-calls and $A^T$-calls $N_A$ and CPU time (in seconds) averaged over $10$ independent runs with dynamic range $60 dB$}
\setlength{\tabcolsep}{0.8mm}{
\begin{tabular}{lllllllllll}
\hline
Method&\multicolumn{2}{l}{$\epsilon:10^0$}&\multicolumn{2}{l}{$\epsilon:10^{-1}$}&\multicolumn{2}{l}{$\epsilon:10^{-2}$}&\multicolumn{2}{l}{$\epsilon:10^{-4}$}&\multicolumn{2}{l}{$\epsilon:10^{-6}$}\\
\cline{2-11}
&Time&$N_A$&Time&$N_A$&Time&$N_A$&Time&$N_A$&Time&$N_A$\\
\midrule[1pt]
SNF&$5.12$&$391.8$&$8.28$&$648.8$&$9.86$&$777.6$&$10.44$&$828.2$&$11.13$&$881.0$\\
SNF(aCG)&$5.05$&$391.8$&$8.32$&$648.8$&$9.89$&$777.6$&$10.83$&$861.2$&$11.37$&$903.2$\\
ASSN&$3.60$&$295.4$&$5.01$&$416.4$&$5.95$&$492.0$&$6.97$&$\mathbf{582.4}$&$7.66$&$\mathbf{642.4}$\\
FPC-AS&$\mathbf{3.14}$&$\mathbf{232.2}$&$8.89$&$644.0$&$11.61$&$844.4$&$13.80$&$1004.4$&$14.08$&$1031.2$\\
SpaRSA&$5.48$&$561.2$&$5.69$&$598.2$&$6.57$&$683.2$&$7.70$&$797.8$&$8.62$&$900.6$\\
F-PG     &$7.07$ &$638.6$ &$8.77$ &$780.8$ &$10.35$ &$937.2$ &$13.05$ &$1157.2$ &$14.85$ &$1338.0$ \\
FS-PG(3) &$3.53$ &$328.6$ &$4.58$ &$422.0$ &$5.60$ &$506.0$ &$6.83$ &$619.8$ &$7.96$ &$714.6$ \\
FS-PG(5) &$3.49$ &$319.0$ &$4.58$ &$428.6$ &$5.72$ &$520.6$ &$\mathbf{6.71}$ &$612.8$ &$\mathbf{7.58}$ &$695.0$ \\
F-PM     &$3.53$ &$310.4$ &$\mathbf{4.14}$ &$\mathbf{374.0}$ &$\mathbf{5.43}$ &$\mathbf{485.8}$ &$7.98$ &$720.2$ &$9.65$ &$868.6$ \\
FS-PM(3) &$3.76$ &$342.0$ &$4.74$ &$429.8$ &$6.50$ &$584.8$ &$10.52$ &$950.2$ &$13.32$ &$1201.2$ \\
FS-PM(5) &$3.48$ &$307.6$ &$4.19$ &$383.2$ &$5.53$ &$502.4$ &$8.01$ &$703.4$ &$9.33$ &$848.8$ \\
\hline
\end{tabular}}
\end{table}
\begin{table}[!htp]
\caption{Total number of $A$-calls and $A^T$-calls $N_A$ and CPU time (in seconds) averaged over $10$ independent runs with dynamic range $80 dB$}\label{80db}
\setlength{\tabcolsep}{0.8mm}{
\begin{tabular}{lllllllllll}
\hline
Method&\multicolumn{2}{l}{$\epsilon:10^0$}&\multicolumn{2}{l}{$\epsilon:10^{-1}$}&\multicolumn{2}{l}{$\epsilon:10^{-2}$}&\multicolumn{2}{l}{$\epsilon:10^{-4}$}&\multicolumn{2}{l}{$\epsilon:10^{-6}$}\\
\cline{2-11}
&Time&$N_A$&Time&$N_A$&Time&$N_A$&Time&$N_A$&Time&$N_A$\\
\midrule[1pt]
SNF&$7.65$&$591.0$&$10.87$&$841.6$&$12.49$&$978.6$&$13.08$&$1024.8$&$15.89$&$1227.6$\\
SNF(aCG)&$7.58$&$591.0$&$10.78$&$841.6$&$12.44$&$978.6$&$13.30$&$1042.2$&$13.99$&$1105.8$\\
ASSN&$5.96$&$482.8$&$7.47$&$601.0$&$8.39$&$690.6$&$9.52$&$780.6$&$10.32$&$852.6$\\
FPC-AS&$\mathbf{4.28}$&$\mathbf{321.4}$&$8.28$&$611.0$&$10.61$&$788.0$&$11.85$&$883.2$&$12.13$&$902.0$\\
SpaRSA&$5.18$&$543.2$&$6.26$&$665.4$&$7.35$&$763.0$&$8.26$&$871.8$&$8.98$&$942.0$\\
F-PG     &$7.18$ &$642.8$ &$8.90$ &$792.8$ &$10.35$ &$951.0$ &$12.47$ &$1134.8$ &$13.50$ &$1231.6$ \\
FS-PG(3) &$4.85$ &$444.8$ &$6.09$ &$555.4$ &$7.01$ &$649.2$ &$7.76$ &$727.0$ &$8.65$ &$789.2$ \\
FS-PG(5) &$4.30$ &$407.2$ &$5.72$ &$521.6$ &$6.77$ &$625.8$ &$\mathbf{7.64}$ &$\mathbf{702.0}$ &$\mathbf{8.15}$ &$\mathbf{753.2}$ \\
F-PM     &$4.17$ &$388.8$ &$\mathbf{5.26}$ &$\mathbf{463.2}$ &$\mathbf{6.55}$ &$583.2$ &$8.14$ &$729.2$ &$9.06$ &$814.6$ \\
FS-PM(3) &$6.00$ &$533.4$ &$6.87$ &$635.4$ &$8.41$ &$748.4$ &$13.08$ &$1162.8$ &$15.04$ &$1348.4$ \\
FS-PM(5) &$4.99$ &$436.4$ &$5.75$ &$525.0$ &$7.08$ &$639.8$ &$9.51$ &$860.0$ &$10.93$ &$987.0$ \\
\hline
\end{tabular}}
\end{table}

From the numerical results, with the increase of the dynamic range, FS-PG(5) is competitive to ASSN or even outperform ASSN in terms of both cpu time and $N_A$. If only a low precision is required, i.e., $\epsilon=10^0$, FPC-AS has the smallest $N_A$ with dynamic ranges 40dB, 60dB and 80dB. With a relative low precision of $\epsilon$, F-PM achieves better performance than FS-PG(5). Although in one iteration F-PM has to calculate the proximal gradient twice, F-PM performs much better than F-PG. In general, FISC with $r=5$ has better performance than FISC with $r=3$. These observations indicate the strength of SDC in general.

\subsection{Logistic regression}
We consider the $\ell_1$-logistic regression problem
\label{sec:logreg}
\begin{equation}
\label{log_reg}
\min_{\mfx=(\hat\mfx,y)\in\mbR^{n+1}}\frac{1}{N}\sum_{i=1}^N\log(1+\exp(-b_i(\la \mfa_i, \hat\mfx\ra+y))) +\lambda \|\mfx\|_1,
\end{equation}
where data pairs $(\mfa_i, b_i)\in \mbR^n\times \{-1,1\}$, correspond to a given dataset. The regularization parameter $\lambda>0$ controls the level of sparsity of a solution to \eqref{log_reg}. In our numerical experiments, $\lambda$ is set to be $0.001$. 
\subsubsection{Algorithm details and the implementation}

The solvers include: prox-SVRG \citep{pSVRG}, Adagrad \citep{Adagrad} and SGD. We give an overview of the tested methods:
\begin{itemize}
\item \textbf{prox-SVRG} stands for a variance reduced stochastic proximal gradient method. Similar to \citep{assnm}, we substitute the basic variance reduction technique proposed in \citep{pSVRG} with the mini-batch version \eqref{ora_vr}. 
\item \textbf{Adagrad} is a stochastic proximal gradient method with a specific strategy for choosing adaptive step sizes. We use the mini-batch gradient \eqref{sto_ora} as the first-order oracle in our implementation.
\item \textbf{SGD} is a stochastic proximal gradient method. The mini-batch gradient \eqref{sto_ora} is used as the first-order oracle in our implementation. 
\item \textbf{sF-PG/sFS-PG(r)} stands for the stochastic version of FIRE/ FISC using the proximal gradient. The stochastic oracle \eqref{sto_ora} is used. In FISC, we take $r=3$ and $r=7$.
\item \textbf{sFVR-PG/sFSVR-PG(r)} stands for the stochastic version of FIRE/ FISC using the proximal gradient. The variance reduced stochastic oracle \eqref{ora_vr} is used. In FISC, we take $r=3$ and $r=7$.
\end{itemize}
For all solvers, the sample size is fixed to be $|\mcS_k|=\lfloor0.01N\rfloor$. The proximal operator of the $\ell_1$-norm is given in \eqref{l1_prox}. In SVRG, we set $m=200$ in \eqref{ora_vr}; in sFVR-PG/sFSVR-PG, we set $m=20$ in \eqref{ora_vr}. Here we intentionally set a larger $m$ in SVRG because it generates a higher precision solution.

\subsubsection{The numerical comparison}
The tested datasets obtained from \textsf{libsvm} \citep{libsvm} in our numerical comparison are summerized in Table \ref{tab:data}. We add a row of ones into the data-matrix $\mfA=(\mfa_1, \mfa_2, \dots ,\mfa_n)$ as coefficients for the bias term in our linear classifier. The datasets for multi-class classification have been manually divided into two types of features. For instance, the \textsf{MNIST} dataset is used to classify even and odd digits.
\begin{table}[!htp]
\caption{Information of the datasets in $\ell_1$-logistic regression}
\label{tab:data}
\centering
\begin{tabular}{|c|c|c|c|}
\hline
Data Set & Data Points $N$ & Variables $n$ &Density\\\hline
\textsf{rcv1}&$20,242$&$47,236$&$0.16\%$\\\hline
\textsf{CINA}&$16,033$&$132$&$29.56\%$\\\hline
\textsf{MNIST}&$60,000$&$784$&$19.12\%$\\\hline
\textsf{gisette}&$6,000$&$5,000$&$12.97\%$\\\hline
\textsf{mushroom}&8,124&112&18.75\%\\\hline
\textsf{synthetic}&10,000&50&22.12\%\\\hline
\textsf{tfidf}&16,087&150,360&0.83\%\\\hline
\textsf{log1p}&$16,087$&$4,272,227$&$0.14\%$\\\hline
\end{tabular}
\end{table}

The initial step sizes varies for different tested datasets and it determines the performance of solvers. Hence, we chose the initial step size from set $\{2^i| i\in\{-7, -6, \dots,$
$7\}\}$. For each dataset, we ran the algorithms with these different parameters and selected a parameter that ensured the best overall performance. Table \ref{tab:step_size} gives the initial step size over these datasets. For SGD, sF(S)-PG and sF(S)VR-PG, we use a exponentially decaying step size. Namely, we decrease the step size by multiplying $0.85$ in each epoch. For all methods, we choose $\mfx_0=0$ as the initial point. 
\begin{table}
\caption{Initial step sizes}
\label{tab:step_size}
\centering
\setlength{\tabcolsep}{0.8mm}{
\begin{tabular}{|c|c|c|c|c|c|c|c|}
\hline
Solver &prox-SVRG&Adagrad&SGD&sF-PG&sFS-PG&sFVR-PG&sFSVR-PG\\\hline
\textsf{rcv1}&8&$2^{-4}$&32&32&32&8&16\\\hline
\textsf{CINA}&2&$2^{-3}$&8&8&8&2&2\\\hline
\textsf{MNIST}&0.5&$2^{-5}$&1&1&1&0.5&0.5\\\hline
\textsf{gisette}&0.5&$2^{-5}$&2&1&2&0.5&0.5\\\hline
\textsf{mushroom}&128&8&8&128&128&128&128\\\hline
\textsf{synthetic}&2&0.125&4&4&4&2&2\\\hline
\textsf{tfidf}&2&0.25&1&0.25&0.5&0.25&0.25\\\hline
\textsf{log1p}&32&0.5&16&16&16&32&32\\\hline
\end{tabular}
}
\end{table}

We next show the performance of all methods. The change of the \textit{relative error} $(f(\mfx)-f(\mfx^*))/(\max\{1, |f(\mfx^*)|\})$ is reported with respect to \textit{epochs}.
Here $x^*$ is a reference solution of problem \eqref{log_reg} generated by S2N-D in \citep{assnm} with a stopping criterion $\|\mfx_k-\mfx_{k-1}\|<10^{-12}$. The numerical results are plotted in Figure \ref{epoch_re}
. We average the results over $10$ independent runs except that only one run is used for \textsf{log1p} because the execution time is too long.

\begin{figure}[tbhp]
\centering
\subfloat[\textsf{rcv1}]{\includegraphics[width=5.5cm]{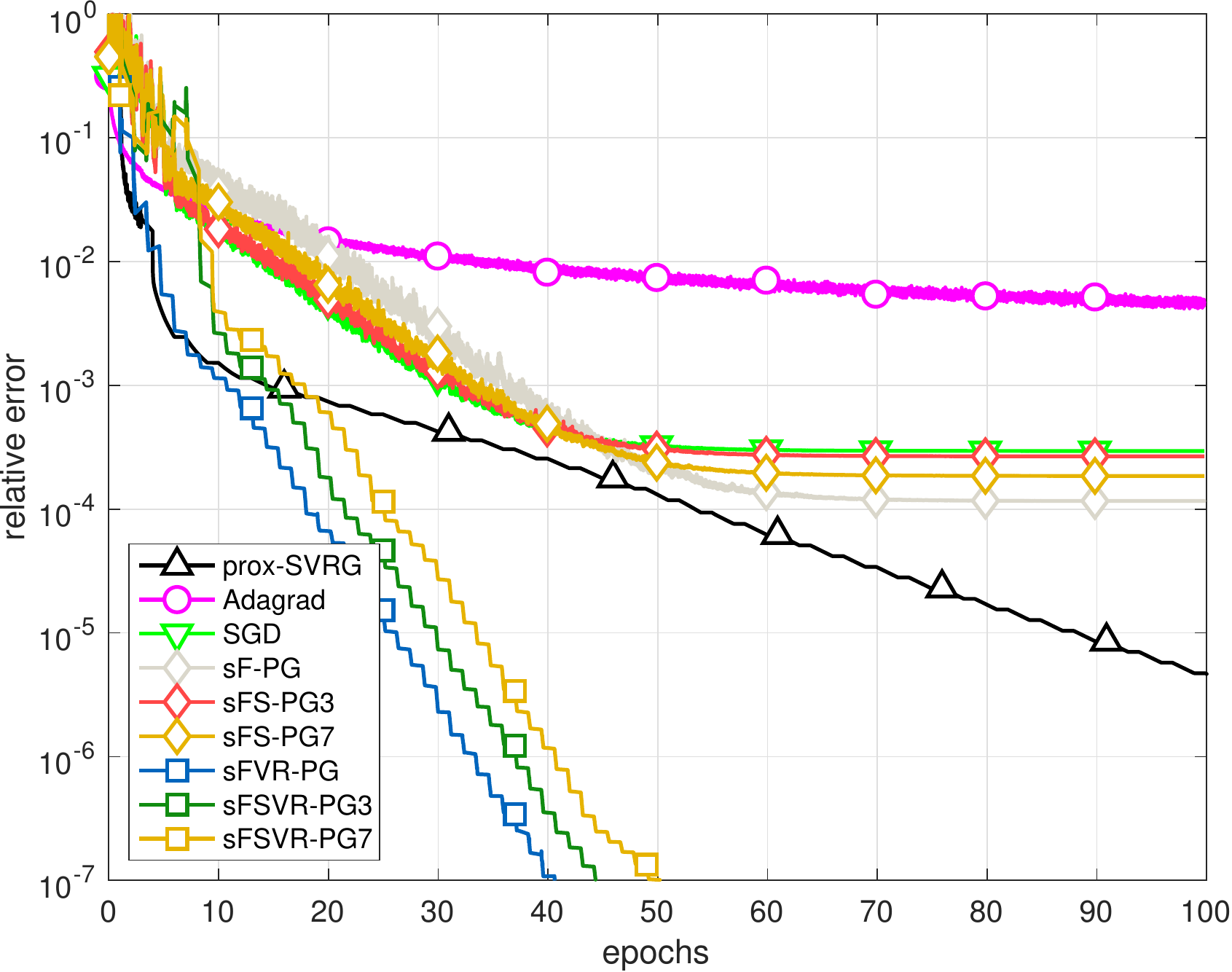}}
\subfloat[\textsf{CINA}]{\includegraphics[width=5.5cm]{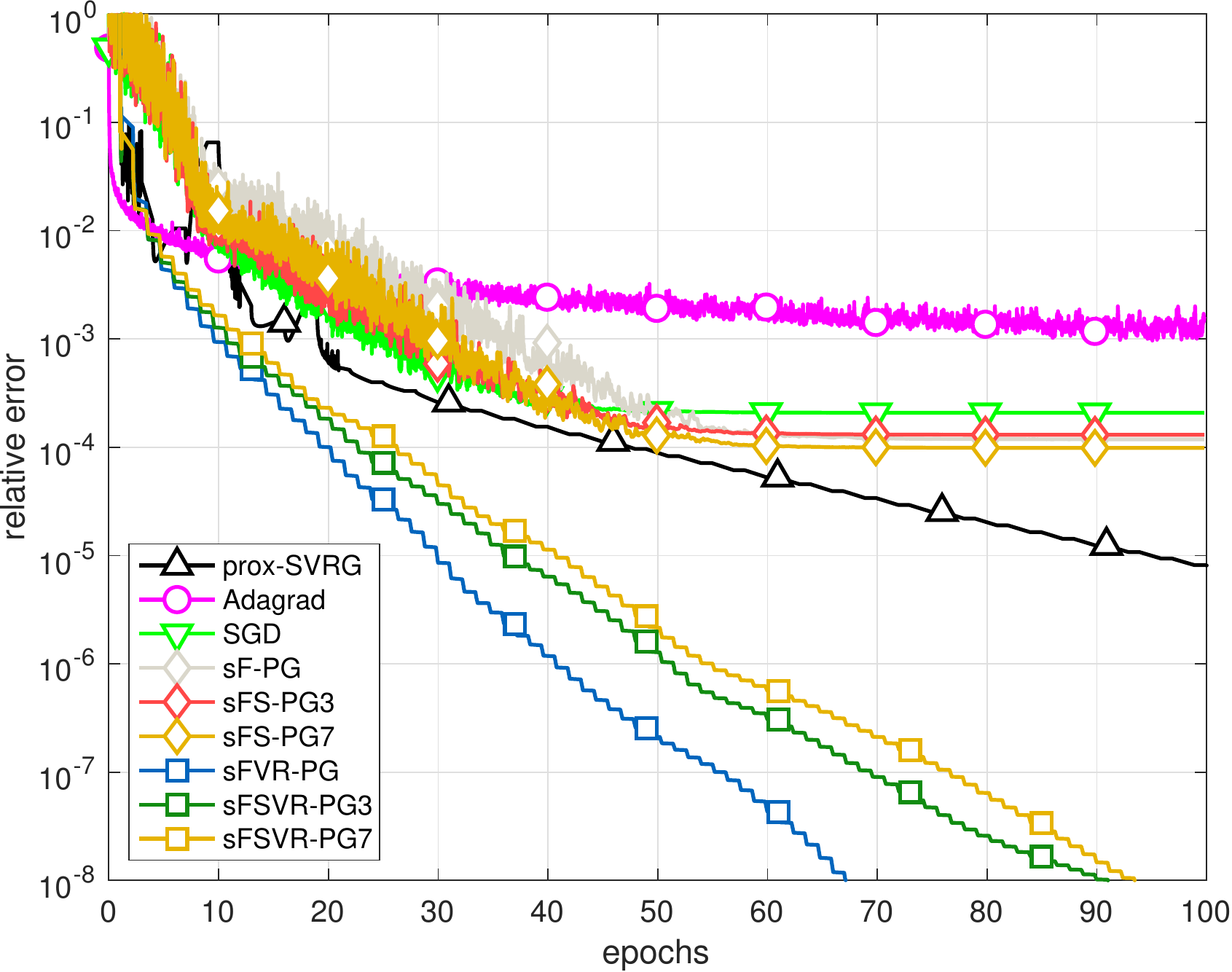}}
\hfill
\subfloat[\textsf{MNIST}]{\includegraphics[width=5.5cm]{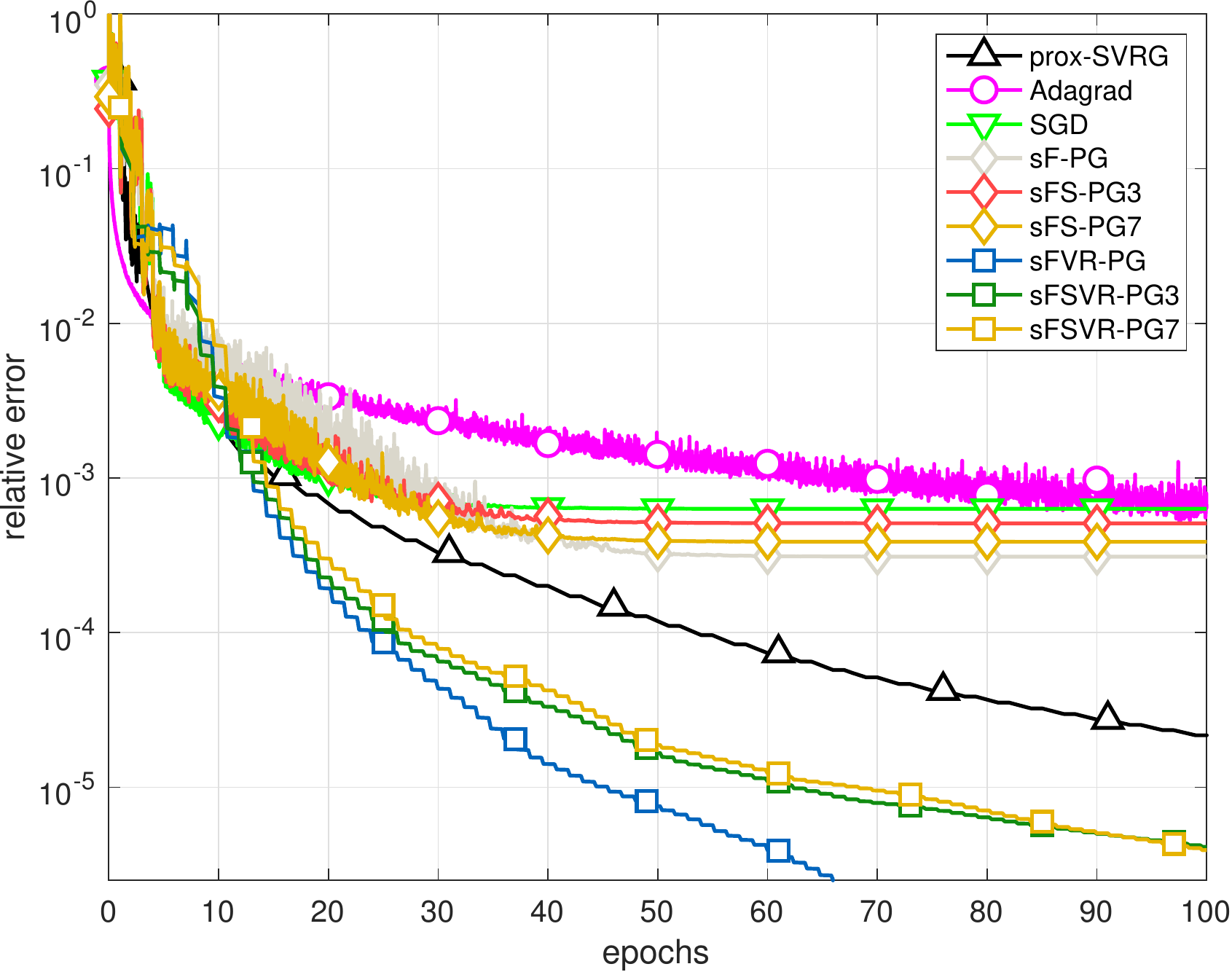}}
\subfloat[\textsf{gisette}]{\includegraphics[width=5.5cm]{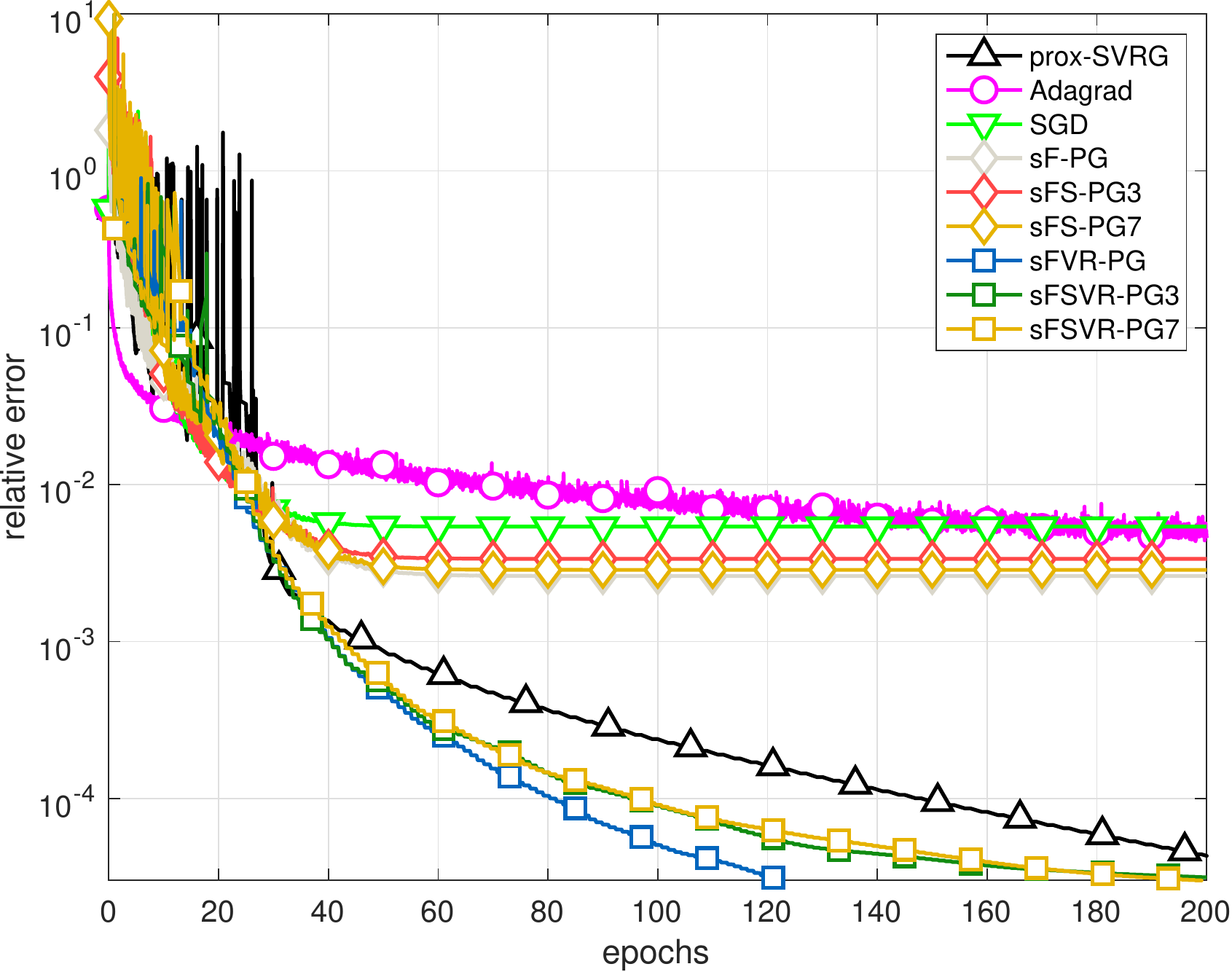}}
\hfill
\subfloat[\textsf{mushroom}]{\includegraphics[width=5.5cm]{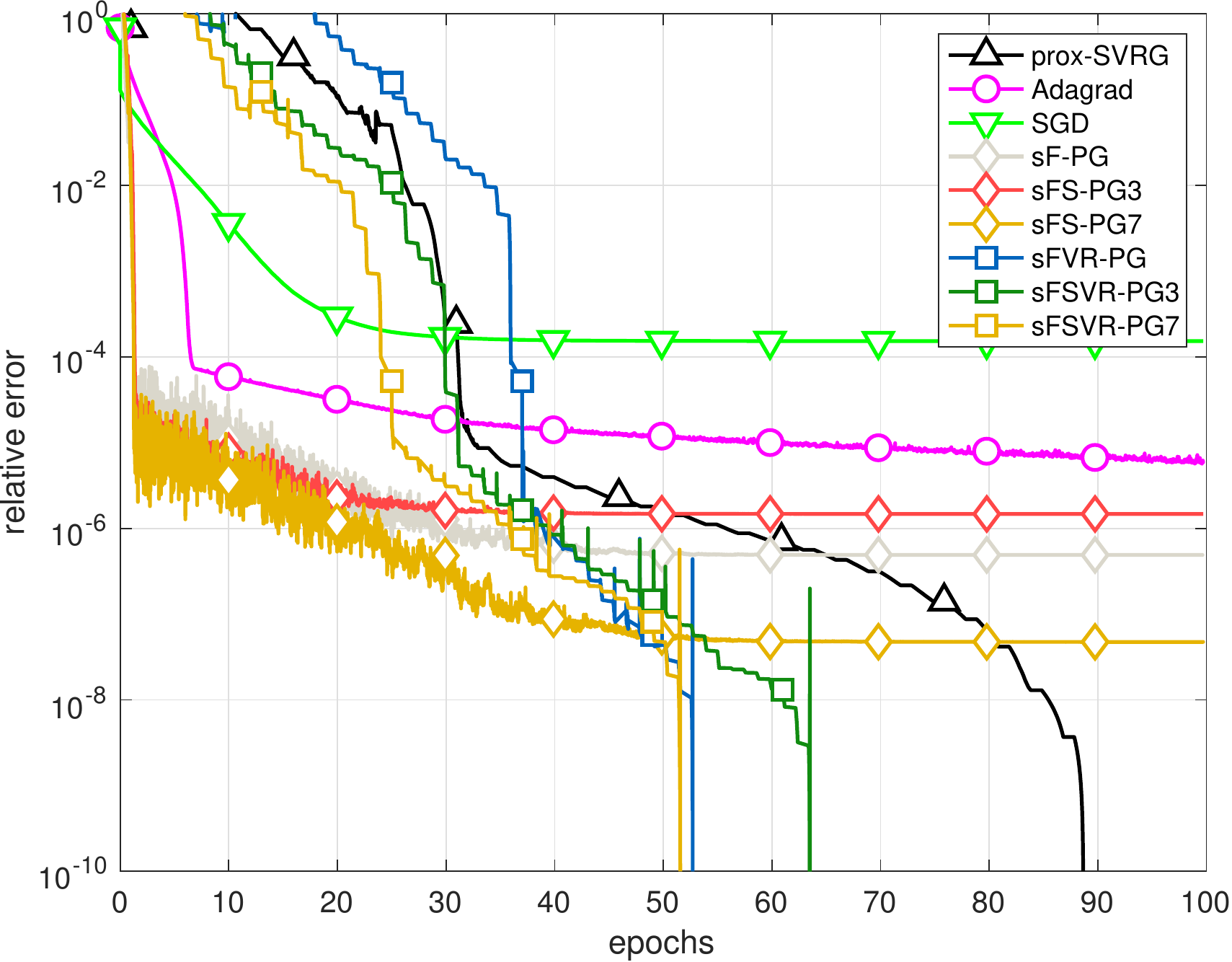}}
\subfloat[\textsf{synthetic}]{\includegraphics[width=5.5cm]{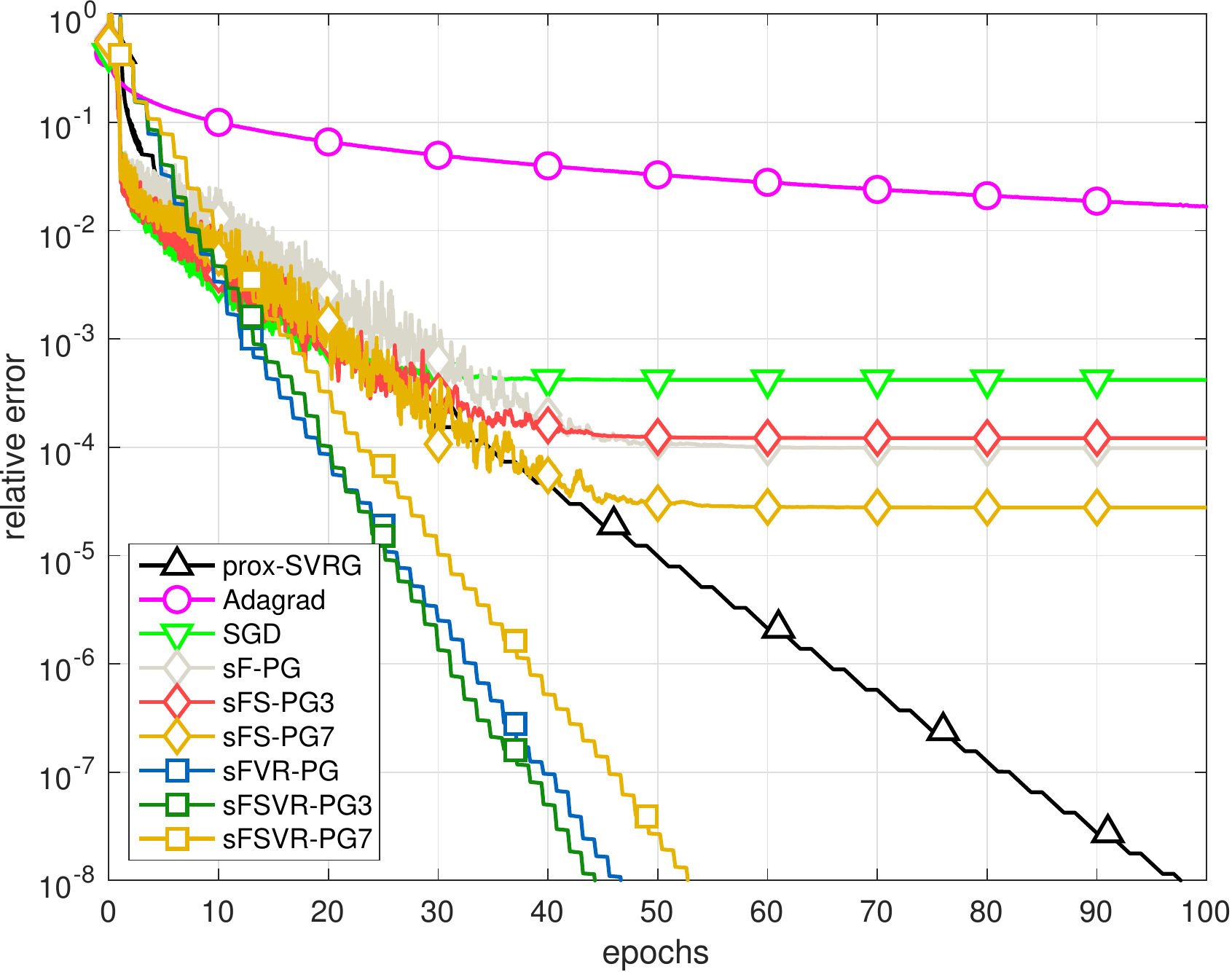}}
\hfill
\subfloat[\textsf{tfidf}]{\includegraphics[width=5.5cm]{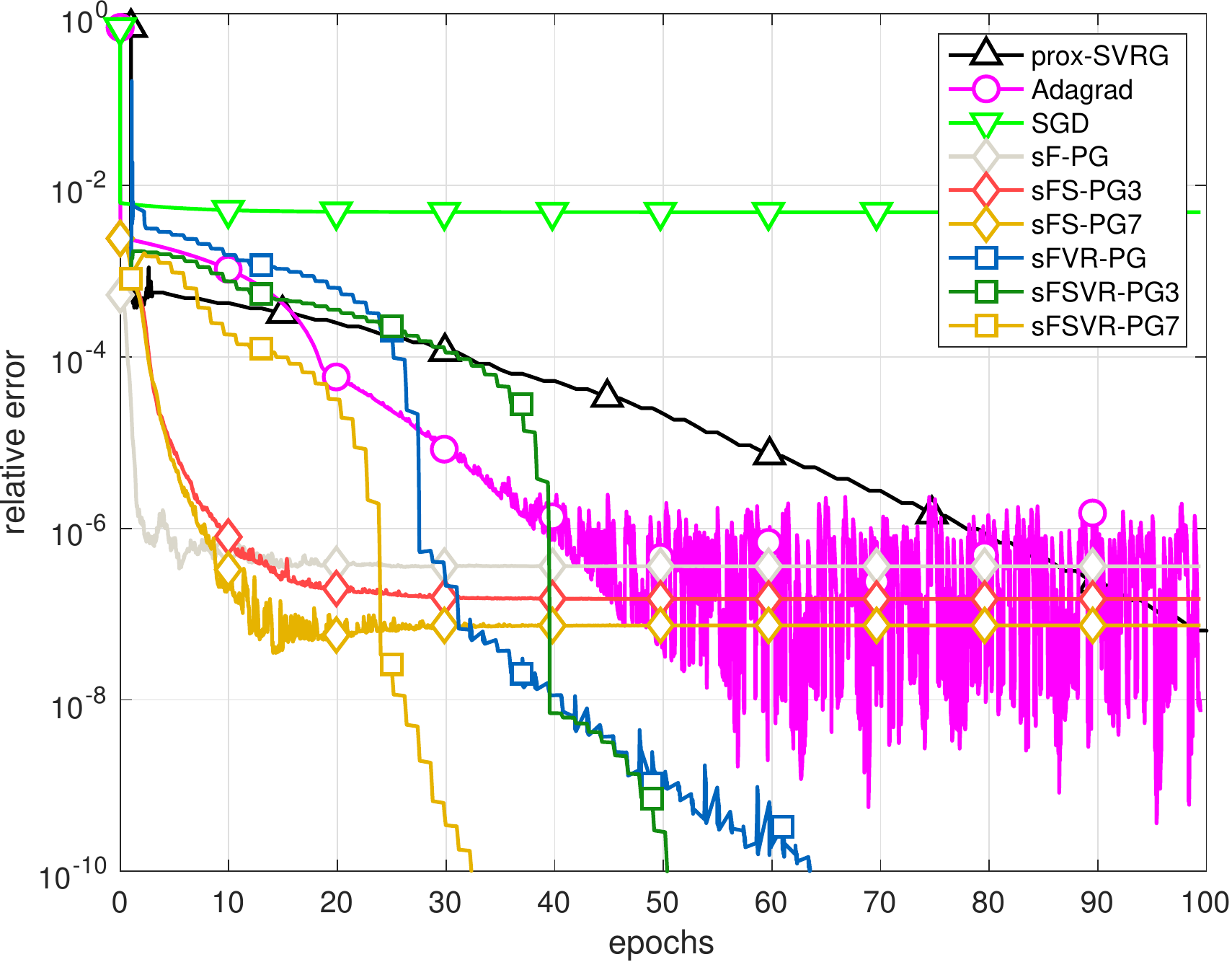}}
\subfloat[\textsf{log1p}]{\includegraphics[width=5.5cm]{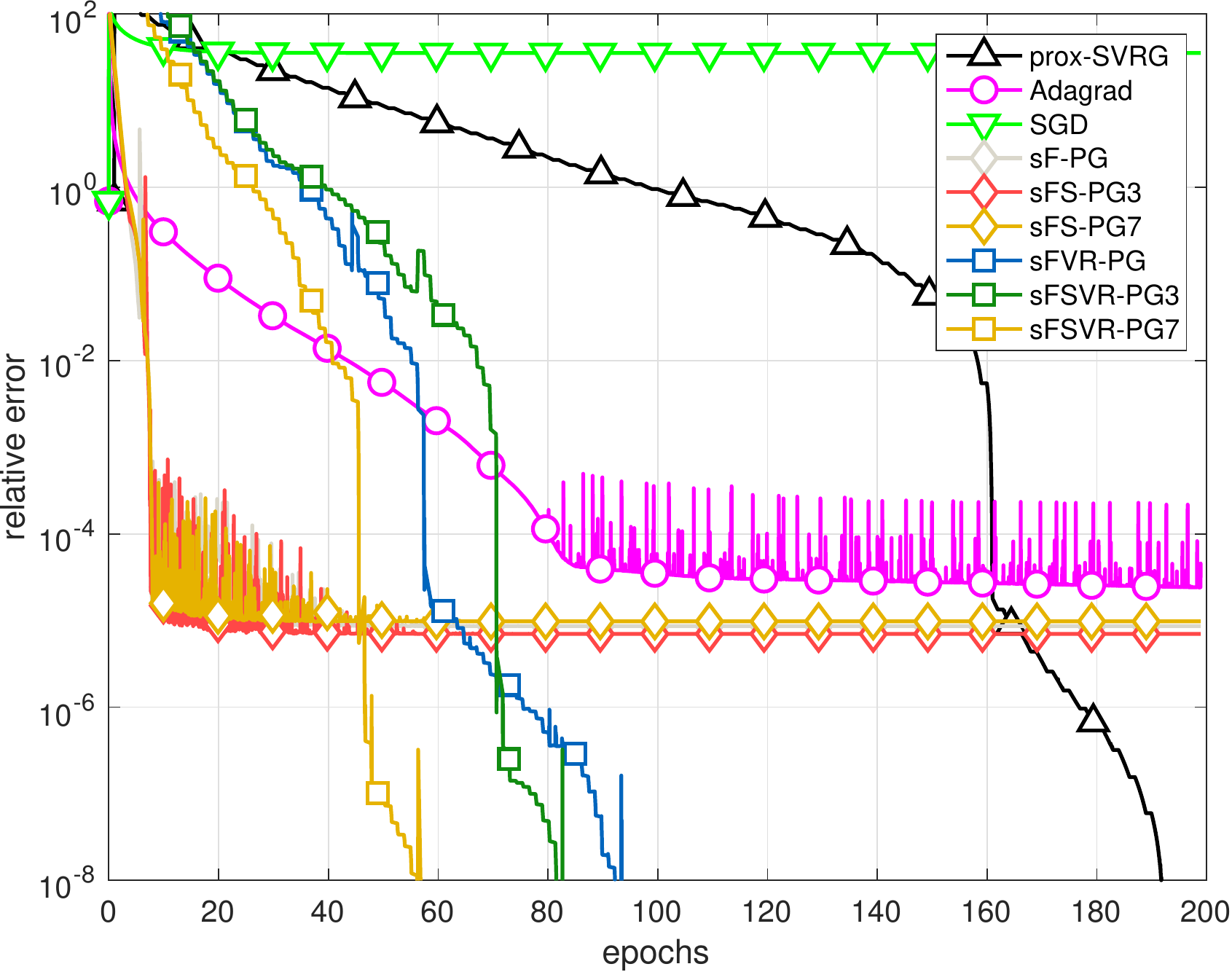}}
\caption{Change of the relative error with respect to the \textit{epochs} for solving the $\ell_1$-logistic regression problem. (Averaged over 10 independent runs, except for \textsf{log1p})}
\label{epoch_re}
\end{figure}

In Figure \ref{epoch_re}, we can roughly split these stochastic methods into two categories: with and without variance reduction techniques. The first category includes sFVR-PG, sFSVR-PG and prox-SVRG, while the second category consists of sF-PG, sFS-PG, SGD and Adagrad. For methods in the first category, we observe that sFVR-PG and sFSVR-PG defeat all other methods, especially in \textit{cpu-time}. sFSVR-PG(7) has competitive performance compared to sFSVR-PG and sFSVR-PG(3). The variance reduction technique seems to be especially well-suited for stochastic FIRE/FISC. On \textsf{log1p}, SVRG decreases slowly in the early stage of the iteration but converges rapidly when the iterates are close to an optimal solution. 

On most test cases, sFS-PG(7) achieves the best performance both with respect to \textit{relative error} and \textit{cpu-time} among other methods,  when variance reduction techniques are not used. Our observation indicates that methods with SDC, i.e., sF-PG and sFS-PG, outperform SGD and Adagrad. On large datasets, like \textsf{tfidf} and \textsf{log1p}, SGD converges to a solution with low precision. Adagrad experiences oscillation after $100$ epochs. Although sF-PG and sFS-PG experience oscillation at first, they finally converge to a precise solution.

In general, sFS-PG(7) is better than sFS-PG(3) and it has similar performance as sF-PG. While sFVR-PG and sFSVR-PG(3) slightly outperform sFSVR-PG(7) in some test cases, sFSVR-PG(7) can lead to a more accurate solution on datasets such as \textsf{mushroom}, \textsf{tfidf} and \textsf{log1p}. Overall, our numerical results indicate that SDC, especially combined with variance reduction techniques, is very promising. 

\subsection{Deep learning}
The optimization problem in deep learning is
$$
\min _{\mfx\in\mbR^n} \frac{1}{N} \sum_{i=1}^{N} l\left(f\left(\mathbf{a}^{(i)}, \mfx\right), b^{(i)}\right)+\lambda\|\mfx\|_2^2,
$$
where $\mfx$ denotes the parameters for training, data pairs $\{(\mfa^{(i)},b^{(i)})\}$ correspond to a given dataset, $f(\cdot, \mfx)$ represents the function determined by the neural network architecture, $l(\cdot,\cdot)$ denotes the loss function and $\lambda$ is the coefficient of weight decay ($\ell_2$-regularization). 

We evaluate our proposed algorithm on deep learning for image classification tasks using the benchmark datasets: CIFAR-10 and CIFAR-100 \citep{cifar}. CIFAR-10 is a database of images from 10 classes and CIFAR-100 consists of images drawn from 100 classes. Both of them consist of 50,000 training images and 10,000 test images. We normalize the data using the channel means and standard deviations for preprocessing. The neural network architectures include DenseNet121 \citep{dccn} and ResNet34 \citep{drlfi}. The number of parameters is listed in Table \ref{cifar_para}.
\begin{table}
\caption{The number of parameters of DenseNet12/ResNet34 on CIFAR-10/CIFAR-100}
\label{cifar_para}
\centering
\begin{tabular}{|c|c|c|}
\hline
&DenseNet121&ResNet34\\
\hline
CIFAR-10&6,956,298&21,282,122\\\hline
CIFAR-100&7,048,548&21,328,292\\
\hline
\end{tabular}
\end{table}

The implemented algorithms include SGD with momentum (MSGD), Adam \citep{adam}, FIRE and FISC with $r=7$. The initial learning rate for different methods is given in Table \ref{cifar_lr}. On CIFAR-10, we train the network using a batch size $128$ for $200$ epochs. The coefficient $\lambda$ is $5\times 10^{-4}$. The learning rate is decreased $10$ times at epoch $150$. On CIFAR-100, we train the network using a batch size $64$ for $300$ epochs and $\lambda$ is $1\times 10^{-4}$. The learning rate is multiplied by $0.1$ at epoch $150$ and epoch $225$. For both datasets, the momentum factor is $0.9$ in MSGD, FIRE and FISC; $(\beta_1, \beta_2)$ in Adam are $(0.9, 0.999)$ on DenseNet and $(0.99, 0.999)$ on ResNet; $\epsilon$ in Adam is $10^{-8}$. 

\begin{table}
\caption{The initial learning rate.}
\label{cifar_lr}
\centering
\begin{tabular}{|c|c|c|}
\hline
&CIFAR-10&CIFAR-100\\
\hline
MSGD&0.1&0.1\\\hline
Adam&0.001&0.001\\\hline
FIRE&0.01&0.1\\\hline
FISC&0.01&0.1\\
\hline
\end{tabular}
\end{table}

\begin{figure}[htbp]
\centering
\begin{minipage}[t]{0.45\textwidth}
\centering
\includegraphics[width=\linewidth]{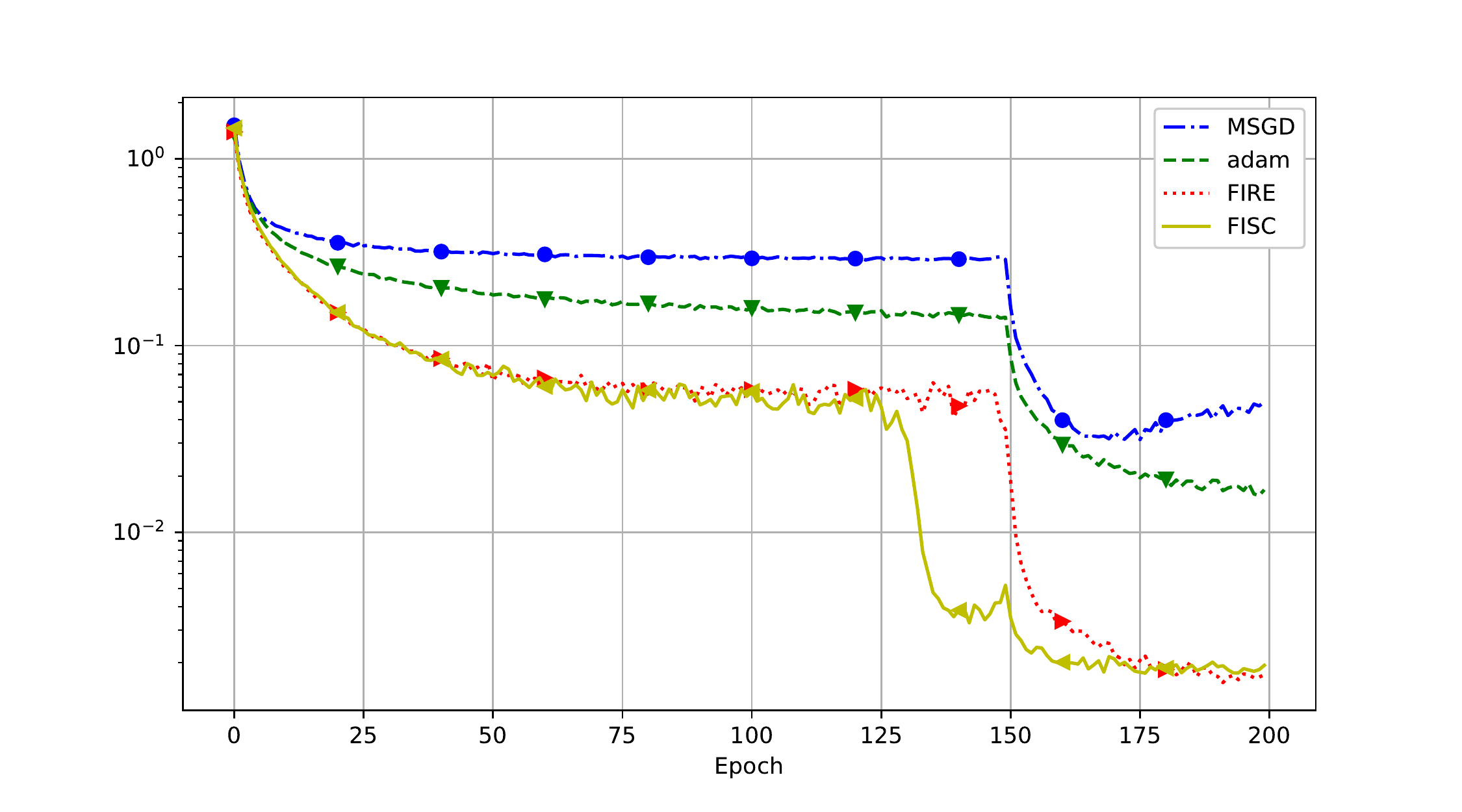}
\caption*{\small{Training Loss, DenseNet121}}
\end{minipage}
\begin{minipage}[t]{0.45\textwidth}
\centering
\includegraphics[width=\linewidth]{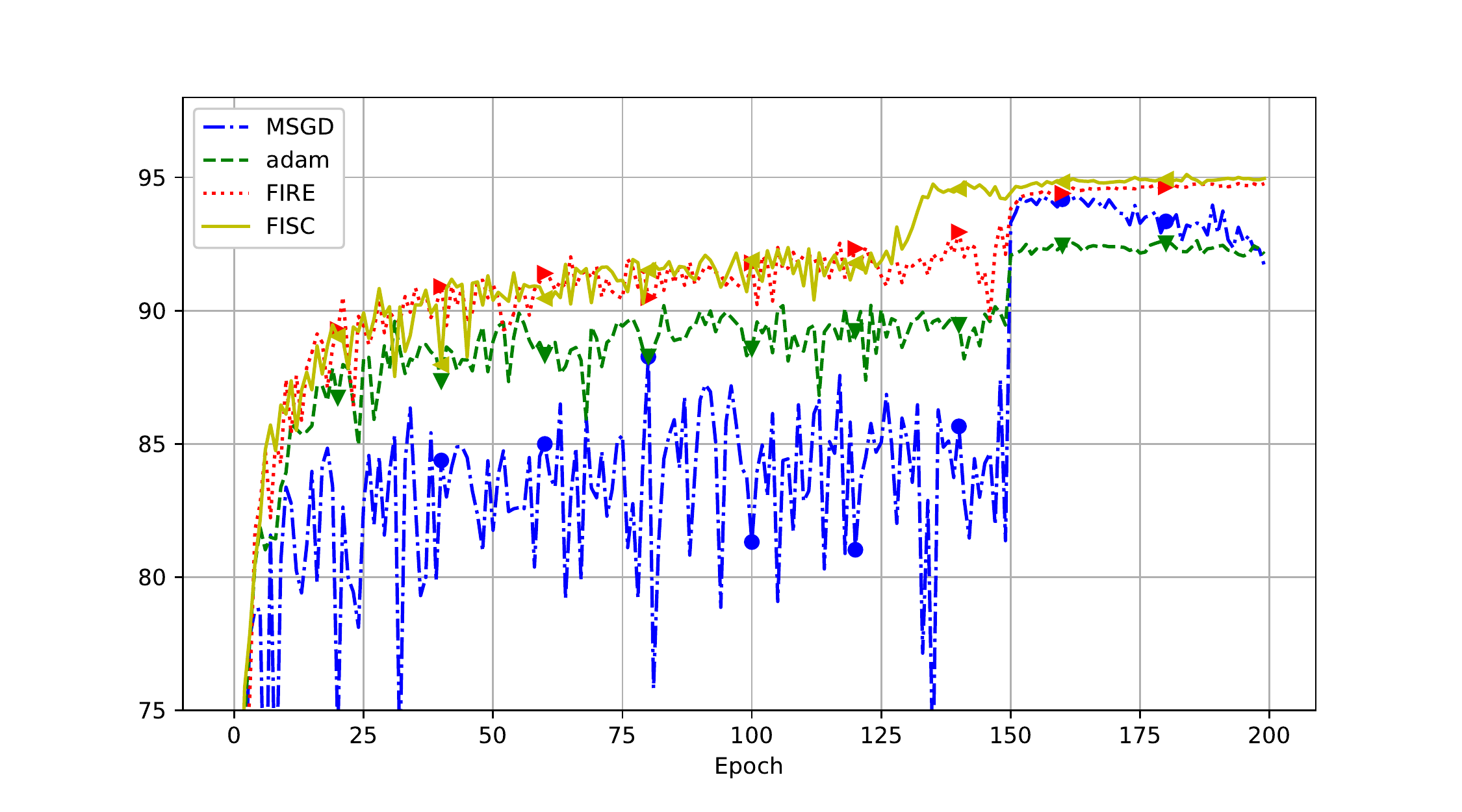}
\caption*{\small{Test Accuracy, DenseNet121}}
\end{minipage}
\hfill
\begin{minipage}[t]{0.45\textwidth}
\centering
\includegraphics[width=\linewidth]{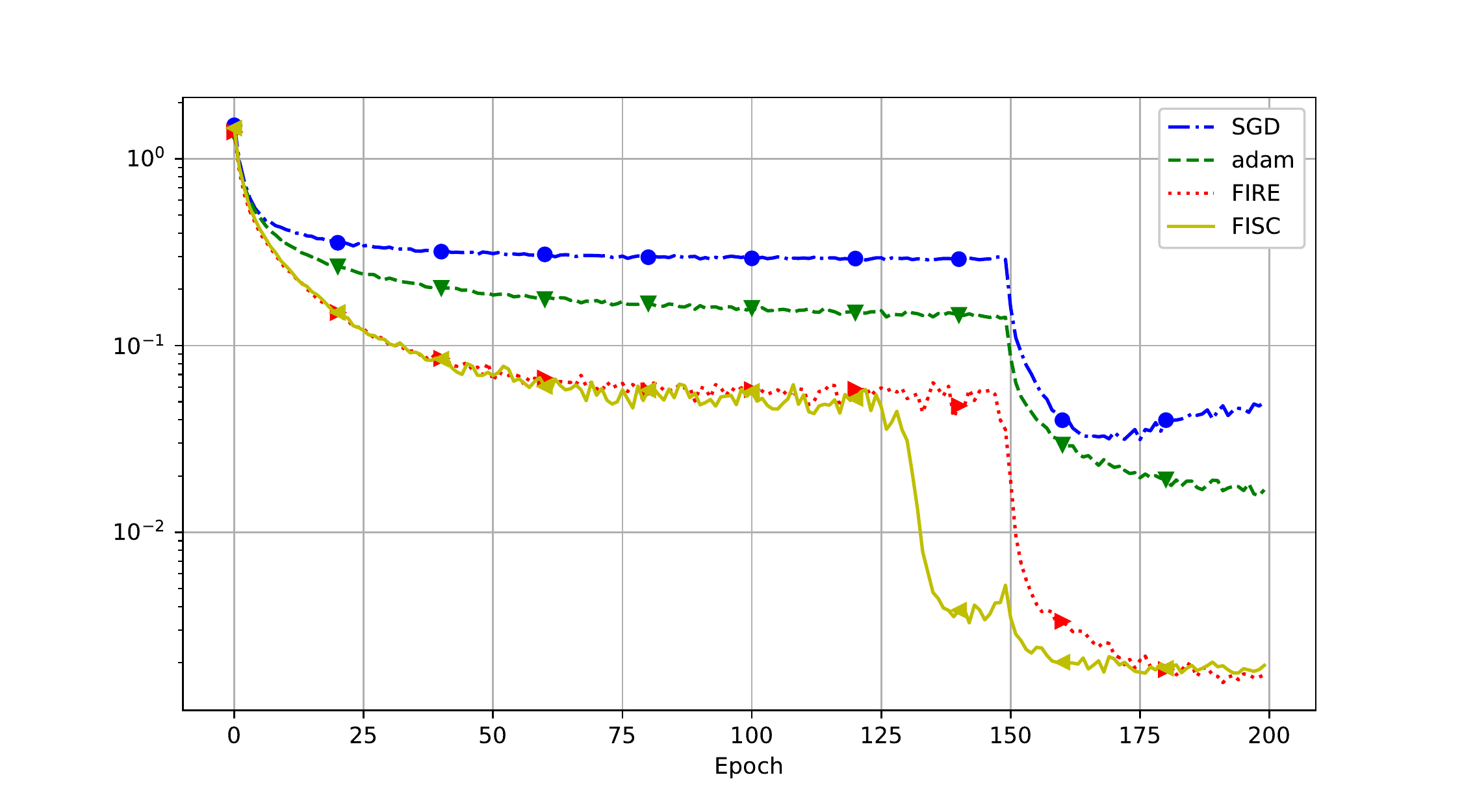}
\caption*{\small{Training Loss, ResNet121}}
\end{minipage}
\begin{minipage}[t]{0.45\textwidth}
\centering
\includegraphics[width=\linewidth]{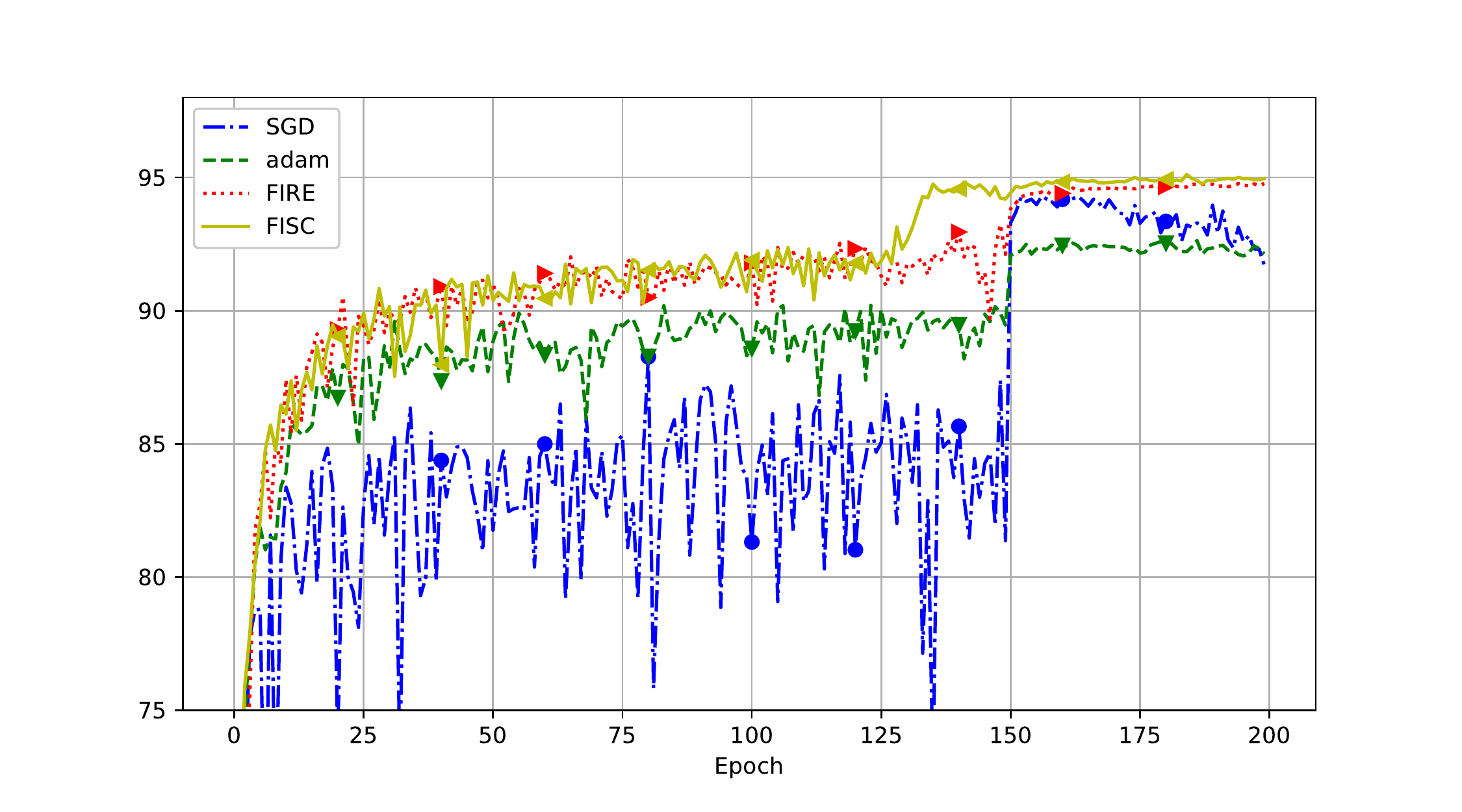}
\caption*{\small{Test Accuracy, ResNet121}}
\end{minipage}
\caption{Numerical results on CIFAR-10. Top: DenseNet121; Bottom: ResNet34.}\label{dense}
\end{figure}

\begin{figure}[htbp]
\centering
\begin{minipage}[t]{0.45\textwidth}
\centering
\includegraphics[width=\linewidth]{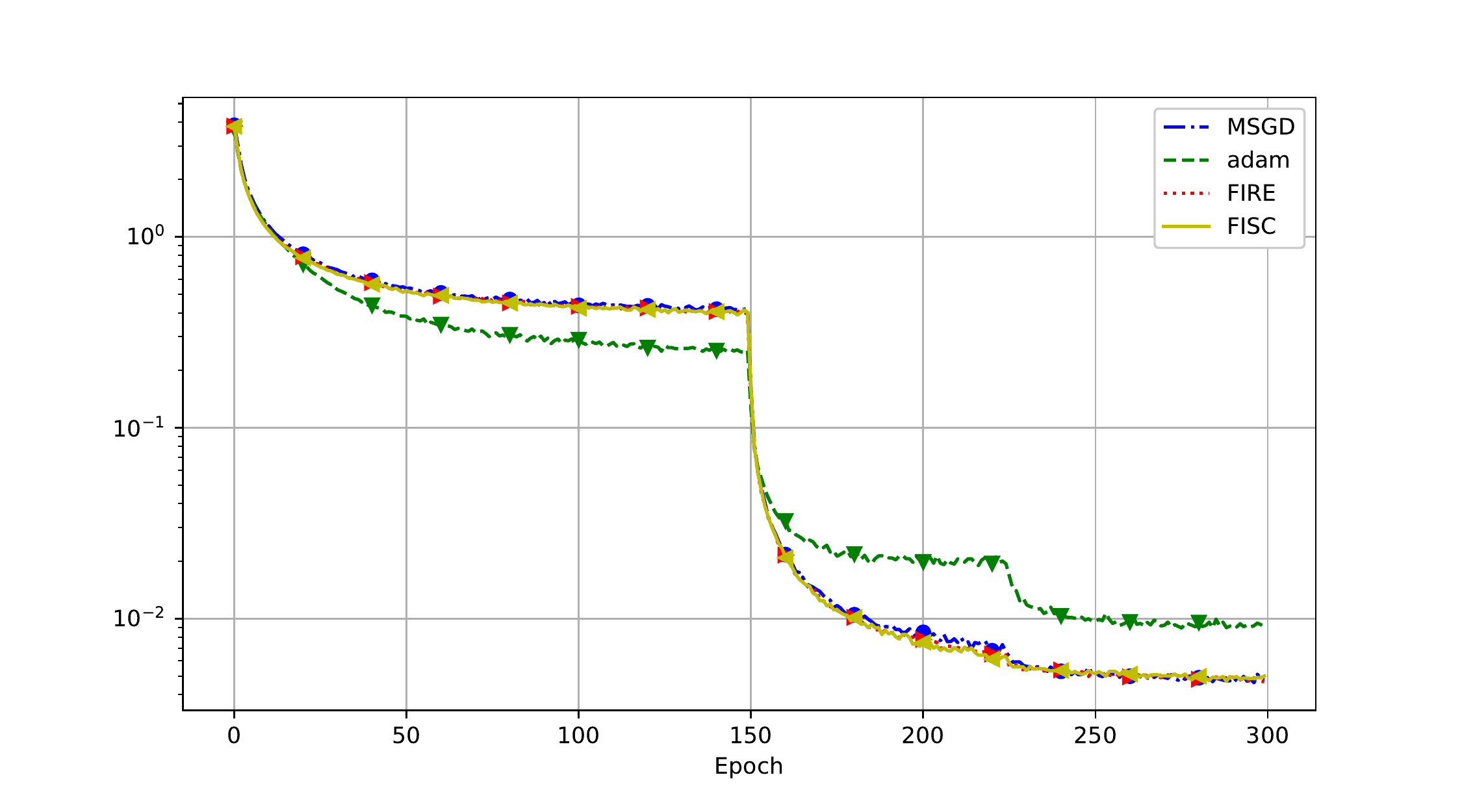}
\caption*{\small{Training Loss, DenseNet121}}
\end{minipage}
\begin{minipage}[t]{0.45\textwidth}
\centering
\includegraphics[width=\linewidth]{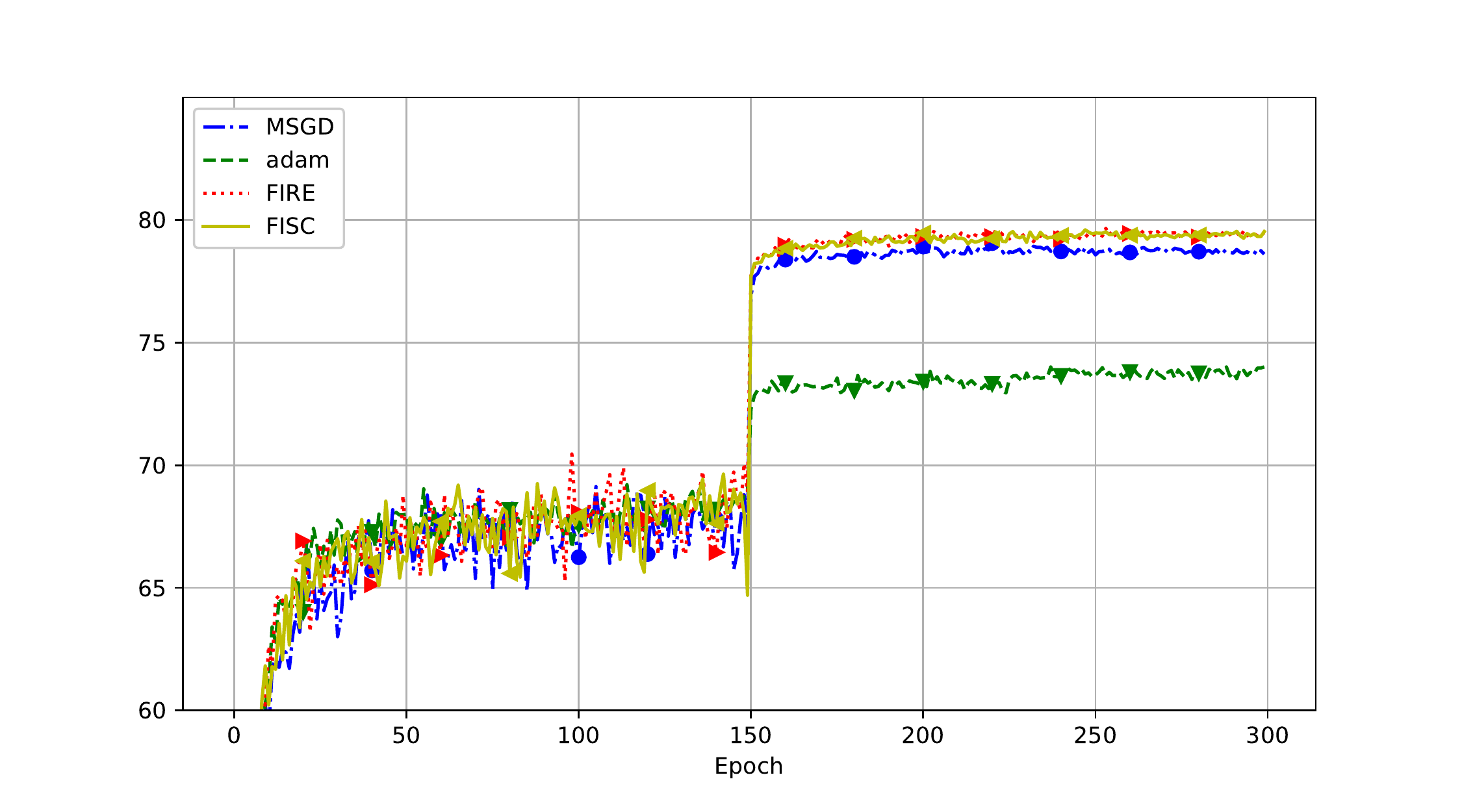}
\caption*{\small{Test Accuracy, DenseNet121}}
\end{minipage}
\hfill
\begin{minipage}[t]{0.45\textwidth}
\centering
\includegraphics[width=\linewidth]{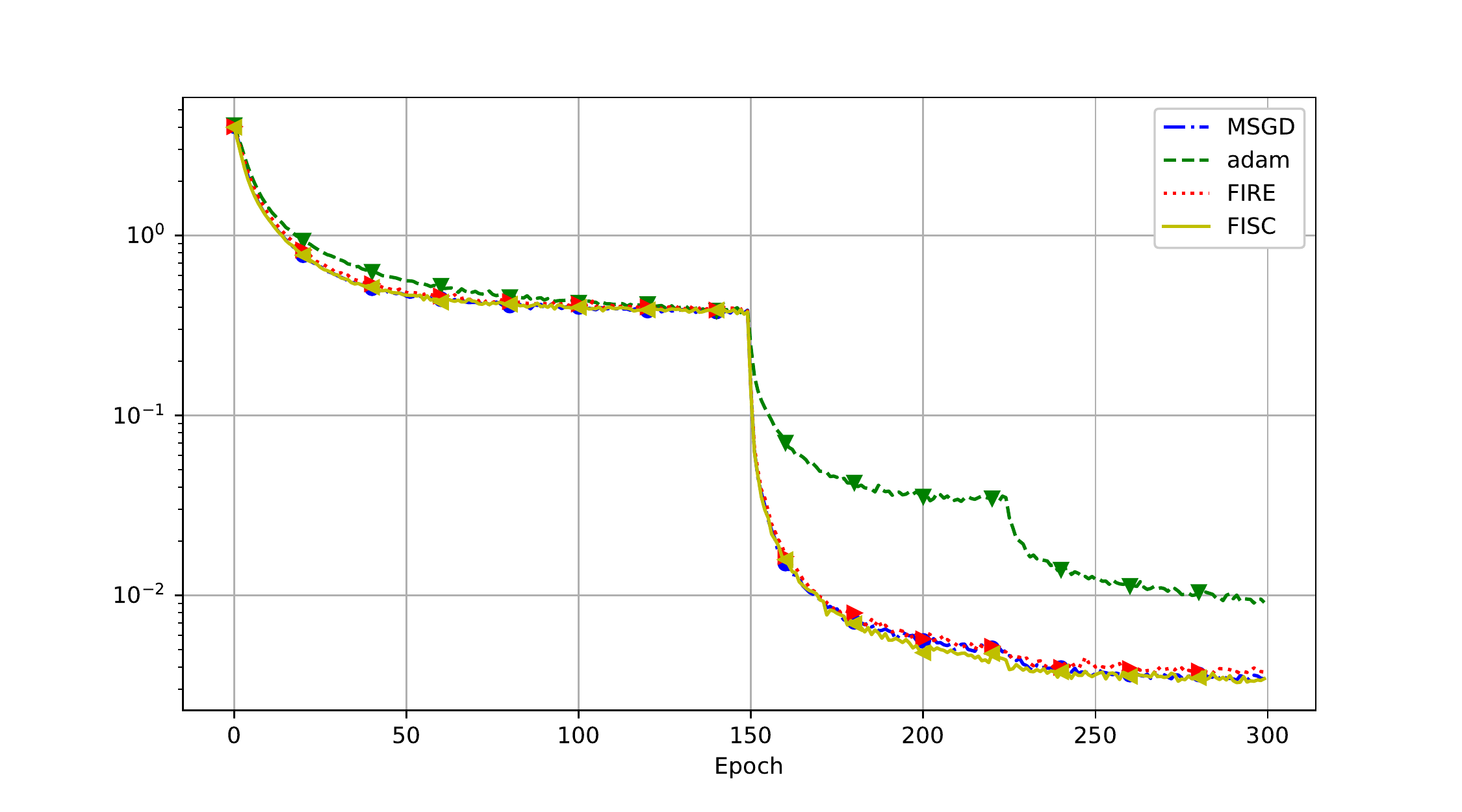}
\caption*{\small{Training Loss, ResNet121}}
\end{minipage}
\begin{minipage}[t]{0.45\textwidth}
\centering
\includegraphics[width=\linewidth]{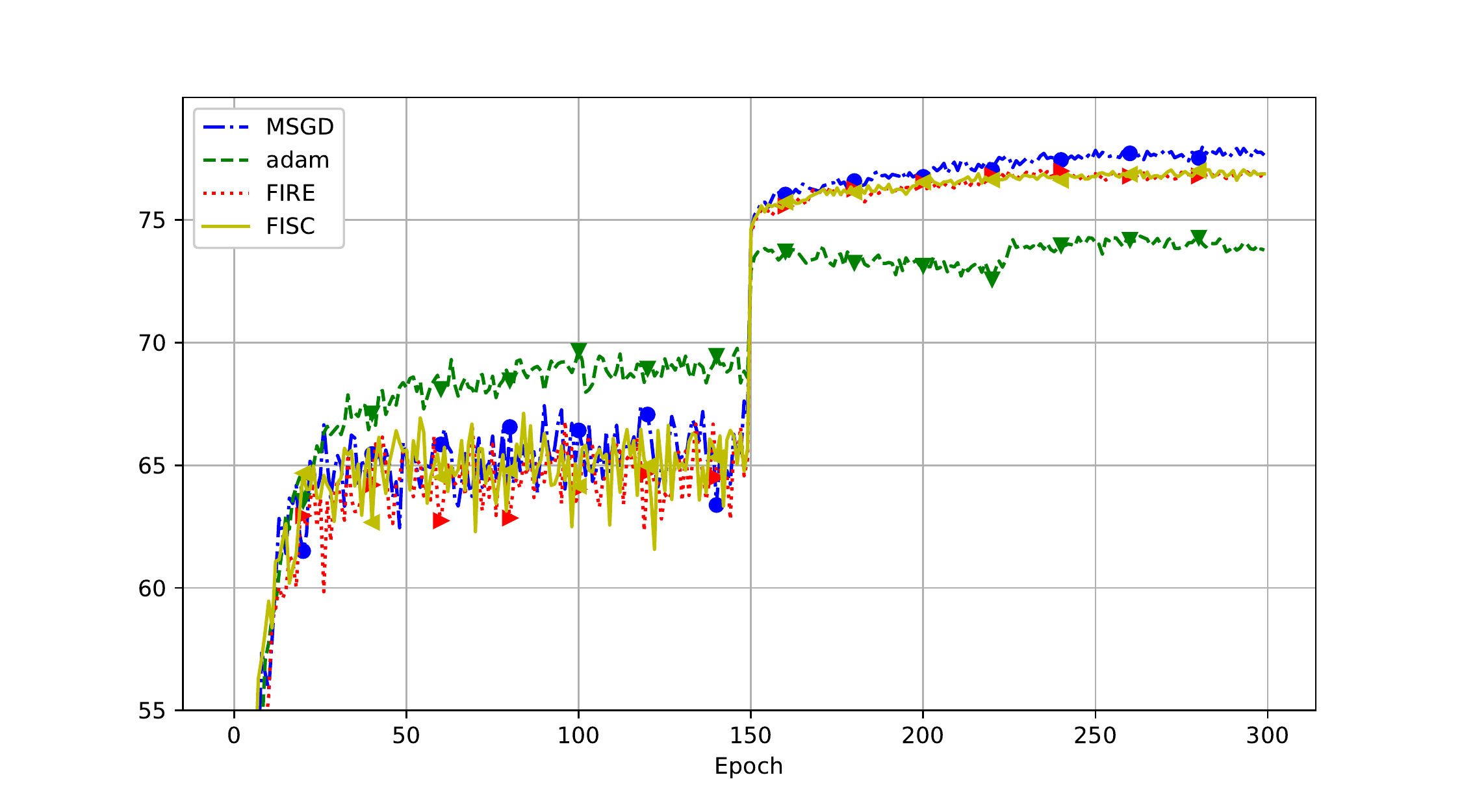}
\caption*{\small{Test Accuracy, ResNet121}}
\end{minipage}
\caption{Numerical results on CIFAR-100. Top: DenseNet121; Bottom: ResNet34.}\label{res}
\end{figure}

Figure \ref{dense} and \ref{res} show that on CIFAR-10, FISC and FIRE have better performance than MSGD and Adam from the very beginning, especially in training loss. On CIFAR-10 with DenseNet, the test accuracy of FISC approaches $95\%$ around epoch 130. On CIFAR-100 with DenseNet, FISC and FIRE outperform MSGD and Adam in test accuracy. This further illustrates the strength of SDC.

\section{Conclusion}
In this paper, we propose a family of first-order methods with SDC. The restarting criterion is the foundation for the global converge of methods with SDC. From an ODE perspective, we construct the FISC-ODE with an $\mcO(t^{-2})$ convergence rate. FISC-PG shows excellent performance in numerical experiments, while FISC-PM has a provable $\mcO(k^{-2})$ convergence rate. Numerical experiments indicate that our algorithmic framework with SDC is competitive and promising. 

\acks{Zaiwen would like to thank Lin Lin and Chao Yang for the kind introduction to and discussion on the FIRE method. Yifei and Zeyu's work is supported in part by the elite undergraduate training program from the School of Mathematical Sciences at Peking University.  Zaiwen's work is supported in part by the NSFC grants 11421101 and 11831002, and by the National Basic Research Project under the grant 2015CB856002.
}

\vskip 0.2in
\bibliography{DeepL}

\begin{thebibliography}{29}
\providecommand{\natexlab}[1]{#1}
\providecommand{\url}[1]{\texttt{#1}}
\expandafter\ifx\csname urlstyle\endcsname\relax
  \providecommand{\doi}[1]{doi: #1}\else
  \providecommand{\doi}{doi: \begingroup \urlstyle{rm}\Url}\fi

\bibitem[Barzilai and Borwein(1998)]{BBstep}
J.~Barzilai and J.~M. Borwein.
\newblock Two-point step size gradient methods.
\newblock \emph{IAM J. Numer. Anal.}, \penalty0 (141--148), 1998.

\bibitem[Beck and Teboulle(2009)]{FISTA}
A.~Beck and M.~Teboulle.
\newblock A fast iterative shrinkage-thresholding algorithm for linear inverse
  problems.
\newblock \emph{SIAM J. on Imaging Sciences}, 2009.

\bibitem[Bitzek et~al.(2006)Bitzek, Koskinen, G{\"a}hler, Moseler, and
  Gumbsch]{FIRE}
Erik Bitzek, Pekka Koskinen, Franz G{\"a}hler, Michael Moseler, and Peter
  Gumbsch.
\newblock Structural relaxation made simple.
\newblock \emph{Physical Review Letters}, 97\penalty0 (17), 2006.

\bibitem[Chih-Chung and Chih-Jen(2011)]{libsvm}
Chang Chih-Chung and Lin Chih-Jen.
\newblock Libsvm: a library for support vector machines.
\newblock \emph{ACM Transactions on Intelligent Systems and Technology}, 2011.

\bibitem[Dai(2000)]{ncgm}
YuHong Dai.
\newblock Nonlinear conjugate gradient methods.
\newblock \emph{Shanghai Science and Technology Publisher}, 2000.

\bibitem[Defazio et~al.(2014)Defazio, Bach, and Lacoste-Julien]{saga}
Aaron Defazio, Francis Bach, and Simon Lacoste-Julien.
\newblock Saga: A fast incremental gradient method with support for
  non-strongly convex composite objectives.
\newblock \emph{NIPS}, 2014.

\bibitem[Duchi et~al.(2011)Duchi, Hazan, and Singer]{Adagrad}
John Duchi, Elad Hazan, and Yoram Singer.
\newblock Adaptive subgradient methods for online learning and stochastic
  optimization.
\newblock \emph{Journal of Machine Learning Research}, 12\penalty0
  (Jul):\penalty0 2121--2159, 2011.

\bibitem[He et~al.(2016)He, Zhang, Ren, and Sun]{drlfi}
Kaiming He, Xiangyu Zhang, Shaoqing Ren, and Jian Sun.
\newblock Deep residual learning for image recognition.
\newblock \emph{CVPR}, 2016.

\bibitem[Huang et~al.(2017)Huang, Liu, van~der Maaten, and Weinberger]{dccn}
Gao Huang, Zhuang Liu, Laurens van~der Maaten, and Kilian~Q. Weinberger.
\newblock Densely connected convolutional networks.
\newblock \emph{CVPR}, 2017.

\bibitem[Johnson and Zhang(2013)]{svrg}
Rie Johnson and Tong Zhang.
\newblock Accelerating stochastic gradient descent using predictive variance
  reduction.
\newblock \emph{NIPS}, 2013.

\bibitem[Krizhevsky(2009)]{cifar}
Alex Krizhevsky.
\newblock \emph{Learning Multiple Layers of Features from Tiny Images}.
\newblock Master's thesis, Department of Computer Science, University of
  Toronto, 2009.

\bibitem[Liu and Nocedal(1989)]{LBFGS}
Dong~C. Liu and Jorge Nocedal.
\newblock On the limited memory bfgs method for large scale optimization.
\newblock \emph{Mathematical Programming}, 45\penalty0 (1):\penalty0 503--528,
  Aug 1989.
\newblock ISSN 1436-4646.
\newblock \doi{10.1007/BF01589116}.
\newblock URL \url{https://doi.org/10.1007/BF01589116}.

\bibitem[Milzarek and Ulbrich(2014)]{asnmw}
A.~Milzarek and M.~Ulbrich.
\newblock A semismooth newton method with multidimensional filter globalization
  for l1-optimization.
\newblock \emph{SIAM Journal on Optimization}, 24:\penalty0 298--333, 2014.

\bibitem[Milzarek et~al.(2018)Milzarek, Xiao, Cen, Wen, and Ulbrich]{assnm}
Andre Milzarek, Xiantao Xiao, Shicong Cen, Zaiwen Wen, and Michael Ulbrich.
\newblock A stochastic semismooth newton method for nonsmooth nonconvex
  optimization.
\newblock \emph{arXiv preprint arXiv:1803.03466}, 2018.

\bibitem[Nesterov(1983)]{Nesterov}
Y.~Nesterov.
\newblock A method of solving a convex programming problem with convergence
  rate {$O(1/k^2)$}.
\newblock \emph{Soviet Mathematics Doklady}, 27\penalty0 (2):\penalty0
  372--376, 1983.

\bibitem[Nesterov(2013)]{iloco}
Yurii Nesterov.
\newblock \emph{Introductory lectures on convex optimization: A basic course},
  volume~87.
\newblock Springer Science {\&} Business Media, 2013.

\bibitem[O'Donoghue and Cand{\'e}s(2013)]{arfag}
B.~O'Donoghue and E.~J. Cand{\'e}s.
\newblock Adaptive restart for accelerated gradient schemes.
\newblock \emph{Found. Comput. Math.}, 2013.

\bibitem[P. and Ba(2015)]{adam}
Kingma~D. P. and J.~L. Ba.
\newblock Adam: a method for stochastic optimization.
\newblock \emph{International Conference on Learning Representations}, pages
  1--11, 2015.

\bibitem[Polyak(1987)]{ito}
B.~T. Polyak.
\newblock \emph{Introduction to Optimization}.
\newblock Optimization Software Inc., 1987.

\bibitem[Schmidt et~al.(2013)Schmidt, LeRoux, and Bach]{sag}
Mark Schmidt, Nicolas LeRoux, and Francis Bach.
\newblock Minimizing finite sums with the stochastic average gradient.
\newblock \emph{Technical report, INRIA}, 2013.

\bibitem[Su et~al.(2016)Su, Boyd., and Cand{\'e}s]{adefm}
Weijie Su, Stephen Boyd., and J~Cand{\'e}s, Emmanuel.
\newblock A differential equation for modeling nesterov's accelerated gradient
  method: Theory and insights.
\newblock \emph{JMLR}, 2016.

\bibitem[Wen et~al.(2010)Wen, Yin, Goldfarb, and Zhang]{afafs}
Zaiwen Wen, Wotao Yin, W.~Goldfarb, and Ding Zhang.
\newblock A fast algorithm for sparse reconstruction based on shrinkage,
  subspace optimization, and continuation.
\newblock \emph{SIAM J. Sci. Comput}, 32:\penalty0 1832--1857, 2010.

\bibitem[Wibisono et~al.(2016)Wibisono, Wilson, and Jordan]{avpoa}
Andre Wibisono, C.~Wilson, Ashia, and I.~Jordan, Michael.
\newblock A variational perspective on accelerated methods in optimization.
\newblock \emph{{arXiv}:1603.04245}, 2016.

\bibitem[Wilson et~al.(2016)Wilson, Recht, and Jordan]{alaom}
C.~Wilson, Ashia, Benjamin Recht, and I.~Jordan, Michael.
\newblock A lyapunov analysis of momentum methods in optimization.
\newblock \emph{{arXiv}:1611.02635}, 2016.

\bibitem[Wright et~al.(2009)Wright, Nowak, and Figueiredo]{srbsa}
S.J. Wright, R.D. Nowak, and M.A.T. Figueiredo.
\newblock Sparse reconstruction by separable approximation.
\newblock \emph{ISSS Trans. Signal Process}, 57:\penalty0 2479--2493, 2009.

\bibitem[Xiao and Zhang(2014)]{pSVRG}
Lin Xiao and Tong Zhang.
\newblock A proximal stochastic gradient method with progressive variance
  reduction.
\newblock \emph{SIAM Journal on Optimization}, 24\penalty0 (4):\penalty0
  2057--2075, 2014.

\bibitem[Xiao et~al.(2017)Xiao, Li, Wen, and Zhang]{arssn}
Xiantao Xiao, Yongfeng Li, Zaiwen Wen, and Liwei Zhang.
\newblock A regularized semi-smooth newton method with projection steps for
  composite convex programs.
\newblock \emph{Springer Science+Business Media}, 2017.

\bibitem[Zhang and Hager(2004)]{anlst}
Hongchao Zhang and William~W. Hager.
\newblock A nonmonotone line search technique and its application to
  unconstrained optimization.
\newblock \emph{SIAM J. OPTIM}, 14\penalty0 (4):\penalty0 1043--1056, 2004.

\bibitem[Zhang et~al.(2018)Zhang, Mokhtari, Sra, and Jadbabaie]{drkda}
Jingzhao Zhang, Aryan Mokhtari, Suvrit Sra, and Ali Jadbabaie.
\newblock Direct runge-kutta discretization achieves acceleration.
\newblock \emph{{arXiv}: 1805.00521}, 2018.

\end{thebibliography}

\end{document}